\documentclass[12pt]{amsart}
\usepackage{graphicx}
\usepackage{amsmath}
\usepackage{amsfonts}
\usepackage{amssymb}
\usepackage{aliascnt}
\usepackage{setspace}
\usepackage{datetime}
\usepackage{color,enumitem,graphicx}
\usepackage[colorlinks=true,urlcolor=blue,
citecolor=red,linkcolor=blue,linktocpage,pdfpagelabels,
bookmarksnumbered,bookmarksopen]{hyperref}
\usepackage{geometry}
\allowdisplaybreaks[4]

\newtheorem{thm}{Theorem}[section]
\newaliascnt{cor}{thm}
\newtheorem{cor}[cor]{Corollary}
\aliascntresetthe{cor}

\newaliascnt{lem}{thm}
\newtheorem{lem}[lem]{Lemma}
\aliascntresetthe{lem}
\newaliascnt{prop}{thm}
\newtheorem{prop}[prop]{Proposition}
\aliascntresetthe{prop}
\theoremstyle{definition}
\newaliascnt{defn}{thm}
\newtheorem{defn}[defn]{Definition}
\aliascntresetthe{defn}
\theoremstyle{remark}
\newaliascnt{rem}{thm}

\aliascntresetthe{rem}
\numberwithin{equation}{section}

\newcommand{\be}{\begin{equation}}
	\newcommand{\ee}{\end{equation}}

\usepackage[nameinlink,noabbrev]{cleveref}

\crefname{thm}{Theorem}{Theorems}
\crefname{prop}{Proposition}{Propositions}
\crefname{lem}{Lemma}{Lemmas}
\crefname{cor}{Corollary}{Corollaries}
\crefname{defn}{Definition}{Definitions}
\crefname{section}{Section}{Sections}
\crefname{subsection}{Section}{Sections}

\newcommand{\supp}{\operatorname{supp}}
\newcommand{\osc}{\operatorname{osc}}
\newcommand{\dist}{\operatorname{dist}}
\newcommand{\Lip}{\operatorname{Lip}}
\newcommand{\quot}{\mathsf p_L}
\newcommand{\dd}{\,\mathrm d}

\begin{document}

\title[3-fold Symmetric Scale-Invariant Euler Flows] {Gradient Growth and Relaxation to Jump Profiles for 3-fold symmetric Scale-Invariant Euler Flows}

\author{Daomin Cao, Junhong Fan, Guolin Qin}
\address{State Key Laboratory of Mathematical Sciences, Academy of Mathematics and Systems Science, Chinese Academy of Sciences, Beijing 100190, P.R. China and University of Chinese Academy of Sciences, Beijing 100049, P.R. China}
\email{dmcao@amt.ac.cn}

\address{Institute of Applied Mathematics, AMSS, Chinese Academy of Sciences, Beijing 100190, and University of Chinese Academy of Sciences, Beijing 100049, P.R. China}
\email{fanjunhong@amss.ac.cn}

\address{State Key Laboratory of Mathematical Sciences, Academy of Mathematics and Systems Science, Chinese Academy of Sciences, Beijing 100190, P.R. China}
\email{qinguolin18@mails.ucas.ac.cn}

\begin{abstract}
We consider long-time behavior of the zero-homogeneous solutions with 3-fold symmetry to the two-dimensional Euler equation. This is the remaining case in the relaxation theory of Said, Elgindi, and Murray [Ann. Sci. \'{E}c. Norm. Sup\'{e}r. (4) \textbf{58} (2025), no.~4, 943--970], which treats $m$-fold symmetry solutions with $m\geq4$.  We prove that every nonconstant $W^{1,p}$ solution satisfies $\|g_\theta(t)\|_{L^p}\to\infty$ as $t\to\pm\infty$ for $1<p\leq\infty$.  For $p=1$, the total variation is conserved, but the $L\log L$ modular tends to infinity whenever it is initially finite.  If $D_\theta g_0$ is a summable sum of non-atomic one-sign components and atoms, every profile in the two omega-limit sets is a jump profile, and each half-orbit approaches its omega-limit set in $W^{\alpha,r}$ for $\alpha r<1$. For such data, every weak $L^2$ infinite-time limit generates a complete $L^2$-precompact orbit.  This structural assumption on $D_\theta g_0$ is automatic for $C^1$ data. In particular, these results answer the question concerning small-scale creation and compact orbit raised by Drivas and Elgindi [EMS Surv. Math. Sci. \textbf{10} (2023), no.~1, 1--100, Problem~5] for all nonconstant smooth 3-fold symmetric scale-invariant flows. Together with the known theory for $m\geq4$, they cover the full well-posed scale-invariant range $m\geq3$.
\end{abstract}

\maketitle


\section{Introduction and main results}

\subsection{Background and the equation}

The vorticity formulation of the two-dimensional incompressible Euler equation is
\begin{equation}
 \partial_t\omega+u\cdot\nabla\omega=0,
 \qquad
 u=\nabla^\perp\Delta^{-1}\omega.
 \label{eq: 1.1}
\end{equation}
Classical global solvability for smooth two-dimensional Euler flows goes back to H\"older and Wolibner \cite{Holder,Wolibner}.  Yudovich established global well-posedness for bounded vorticity in sufficiently regular bounded planar domains \cite{Yudovich63}.  On $\mathbb R^2$ one usually assumes $\omega_0\in L^1\cap L^\infty$.  Although transport preserves the vorticity amplitude, it does not prevent the creation of small scales.  General growth and instability mechanisms were developed in \cite{Koch,MSY,Yudovich00}.

The hyperbolic construction of Bahouri and Chemin \cite{BC} became a basic mechanism for rapid small-scale creation.  Denisov constructed examples exhibiting infinite-time superlinear growth and arbitrarily long finite-time double-exponential amplification on the torus \cite{Denisov,Denisov2}. Kiselev and \v Sver\'{a}k obtained sharp double-exponential growth in the disk \cite{KS}.  Other examples include fast growth in symmetric smooth domains \cite{Xu}, exponential growth on the torus \cite{Zlatos}, and maximal double-exponential growth on the half-plane \cite{ZlatosHP}.  Chen and Sun obtained linear-in-time filamentation under arbitrarily small smooth nonnegative perturbations of nontrivial, nonnegative, compactly supported half-plane data \cite{ChenSun}.

In boundary-free settings, Choi and Jeong constructed linear gradient growth near the Lamb dipole \cite{CJ}, using its orbital stability \cite{AbeChoi}.  Jeong, Yao, and Zhou obtained superlinear growth on the torus and the plane \cite{JYZ}.  Related small-scale creation results for free-boundary Euler flows and gravity-capillary water waves appear in \cite{HLY,Tu}.

Shnirelman introduced generalized Lyapunov functionals for the infinite-time loss of smoothness \cite{Shnirelman97}.  Twisting provides another robust Lagrangian mechanism \cite{DEJ}, and a quantitative growth result for an initial-data-dependent norm of the Lagrangian flow was proved in \cite{SaidLag}.

Other long-time questions include wandering behavior \cite{Nadirashvili} and non-equilibrium scenarios \cite{Shnirelman}.  See \cite{DE} for a survey. Perturbative relaxation results include nonlinear inviscid damping near planar shear flows \cite{BM,IJCouette,IJdamp,MZ}, axi-symmetrization near a point vortex \cite{IJaxi}, and asymptotic stability at Yudovich regularity in an infinite channel \cite{GuoLuo}.

Following \cite[Definitions~1.3 and~1.4]{EMS}, we call a vorticity $\omega$ scale-invariant if
$$
 \omega(t,\lambda x)=\omega(t,x)
 \qquad\text{for all }\lambda>0,
$$
and $m$-fold symmetric if
$$
 \omega(t,\mathcal O_m x)=\omega(t,x),
$$
where $\mathcal O_m$ denotes the counterclockwise rotation by $2\pi/m$.

A scale-invariant, or equivalently zero-homogeneous, vorticity can be written as $\omega(t,r,\theta)=g(t,\theta)$.  Such a vorticity does not decay at spatial infinity, so discrete rotational symmetry supplies the cancellation needed for the whole-plane formulation.  Elgindi and Jeong proved global well-posedness for bounded $m$-fold symmetric vorticity for $m\geq3$, and derived the scale-invariant equation studied here \cite[Proposition~3.5]{EJ}.  Beyond the Yudovich class, local well-posedness has been established for a class of unbounded vorticities, together with examples that blow up in finite time and leave that well-posedness class \cite{EMSrough}.

For the zero-homogeneous ansatz, write the stream function as $\psi(t,r,\theta)=r^2G(t,\theta)$.  Then
$$
 \Delta\psi=G_{\theta\theta}+4G,
 \qquad
 u_r=-rG_\theta,
 \qquad
 u_\theta=2rG,
$$
and hence
$$
 (\partial_{\theta\theta}+4)G=g,
 \qquad
 g_t+2Gg_\theta=0.
$$
For an $m$-fold symmetric profile the angular frequencies are $mk$ with $ k \in\mathbb{Z} $. The case $m=2$ is resonant with the kernel of $\partial_{\theta\theta}+4$, whereas $m=3$ is the lowest nonresonant symmetry and the first case for which the reduced dynamics are uniquely defined without an additional normalization.

For symmetric vortex patches with corners at the origin, the well-posedness theory was developed in \cite{EJpatchI}, while the long-time dynamics including periodic motion, cusp formation, and spiraling was investigated in \cite{EJpatch}.  A different scale-invariant ansatz based on logarithmic spirals was studied in \cite{JS}. Related vortex sheets are treated in \cite{CKOexist,CKOinst,CKOwell}.

Drivas and Elgindi's Problem~5 asks for generic infinite-time singularity formation and convergence to compact orbits for smooth solutions of the scale-invariant equation \cite[Problem~5]{DE}.  In the compact-orbit conclusion considered here, a weak $L^2$ infinite-time limit is a profile $h$ for which
$$
 g(t_n)\rightharpoonup h\quad\text{weakly in }L^2(\mathbb T_L)
 \qquad\text{along some sequence }|t_n|\to\infty,
$$
and the conclusion is that the complete solution through $h$ has a precompact orbit in $L^2(\mathbb T_L)$.  The small-scale creation and compact-orbit conclusions arise from different mechanisms.  The growth result below is a full two-sided limit rather than a limsup statement, while the compact-orbit conclusion requires passing from convergence at isolated times to compactness along a complete limiting trajectory.

In the breakthrough work \cite{EMS}, for $m\geq4$, Said, Elgindi and Murray established a Lagrangian expansion--contraction principle and a detailed relaxation theory.  For every $m\geq3$, their whole-plane result concerns the $m$-fold symmetric subspace of $C^1(\mathbb R^2)\cap L^1(\mathbb R^2)$, and their spherical result the $m$-fold symmetric subspace of $C^1(\mathbb S^2)$. In each space, the data whose solutions diverge in $C^1$ as $t\to+\infty$ form a dense set of second category.  When $m=3$, however, the kernel underlying the relaxation argument is not sign definite.  As noted in \cite[Remark~2.1]{EMS}, the relaxation statements based on kernel positivity remain valid for 3-fold odd-symmetric data, but not for general 3-fold symmetric data.  In the normalization of \cite[Remark~2.3]{EMS}, the sector kernel on $[-\pi/3,\pi/3]$ is
\begin{equation}
K^3_{1/(\partial_{\theta\theta}+4)}(\theta)
 =\frac{3\pi}{10}\left|\sin\frac{3\theta}{2}\right|-\frac7{20}.
 \label{eq: 1.2}
\end{equation}
Here $K^3_{1/(\partial_{\theta\theta}+4)}$ denotes the 3-fold symmetrized convolution kernel of $(\partial_{\theta\theta}+4)^{-1}$ on $2\pi/3$-periodic profiles. Thus the kernel changes sign if $m=3$.  The four-point cross-ratio monotonicity developed below replaces the kernel positivity used for $m\geq4$.

In the case $m=3$ the positive and negative components of $D_\theta g$ are transported separately, while the cyclic order of their material labels is preserved.  An entropy identity forces two-sided growth of the angular derivative, including the $L\log L$ endpoint.  For relaxation, we examine a four-point cross-ratio in projective coordinates.  Its derivative has a positive Peano kernel representation, which yields the one- and two-atom bounds in \cref{thm: 1.3}.  Together with summability of the component masses, these bounds exclude both absolutely continuous and Cantor variation from the limiting profiles.  Finite decompositions lead to relaxation to profiles with finitely many jumps.  The summability condition is automatic for $C^1$ data.  In that case, compactness of time translates and invariance of the limit sets show that every weak infinite-time limit generates a complete solution with a precompact orbit.

Without symmetry restrictions, Alazard and Said \cite{AS} proved forward-time Sobolev growth on dense $G_\delta$ sets of zero-mean data, in the sense that the relevant Sobolev norms have infinite limsup.  Their results apply in $H^s(\mathbb T^2)$ for $s>4$ and in $H^s(\mathbb R^2)\cap L^1(\mathbb R^2)$ for $s>1$.  Here the angular derivative of every nonconstant $W^{1,p}$ profile, $1<p\leq\infty$, tends to infinity in both time directions.

In this paper, we are concerned with the remaining case $m=3$. For $m=3$, we establish that every nonconstant smooth angular profile has two-sided gradient growth, and every weak $L^2$ infinite-time limit generates an $L^2$-precompact complete orbit.  Thus we give an affirmative answer to Problem 5 for all nonconstant smooth data in the $m=3$ scale-invariant class \cite[Problem~5]{DE}, where its compact-orbit conclusion is the $L^2$-precompactness of the complete orbit through each weak $L^2$ infinite-time limit. Combined with the theory for $m\geq 4$ in \cite{EMS}, the gradient-growth and compact-orbit conclusions hold for every nonconstant smooth angular profile in the well-posed scale-invariant range $m\geq3$.

To be precise, we first introduce some definitions and notations. Set
\begin{equation}
 \mathbb T_L=\mathbb R/L\mathbb Z,
 \qquad
 L=\frac{2\pi}{3}.
 \label{eq: 1.3}
\end{equation}
All functions below are understood to be $L$-periodic in $\theta$. On the fundamental circle, the reduced equation is
\begin{equation}
 g_t+2Gg_\theta=0,
 \qquad
 (\partial_{\theta\theta}+4)G=g,
 \qquad \theta\in\mathbb T_L.
 \label{eq: 1.4}
\end{equation}
The Fourier frequencies on $\mathbb T_L$ are $3k$, $k\in\mathbb Z$. Consequently, for every $s\in\mathbb R$,
$$
 (\partial_{\theta\theta}+4):
 H^{s+2}(\mathbb T_L)\longrightarrow H^s(\mathbb T_L)
$$
is an isomorphism, and its inverse is well defined on periodic distributions. Write
\begin{equation}
 A:=\partial_{\theta\theta}+4,
 \qquad u:=2G,
 \qquad q:=D_\theta g.
 \label{eq: 1.5}
\end{equation}
Here and below, $D=D_\theta$ denotes distributional differentiation on $\mathbb T_L$. Let $\mathcal M(\mathbb T_L)$ denote the space of finite signed Radon measures, and $\|\cdot\|_{\mathcal M}$ be the total-variation norm.  If $T$ is a Borel map and $\mu$ is a measure, $T_\#\mu$ denotes the push-forward,
$$
 (T_\#\mu)(E):=\mu(T^{-1}(E)).
$$
For a finite set $E$, $\#E$ denotes its cardinality. Let $d_L$ denote the geodesic distance on $\mathbb T_L$.  For $0<\alpha<1$ and $1\leq r<\infty$, denote
$$
 [f]_{W^{\alpha,r}(\mathbb T_L)}^r
 :=\int_{\mathbb T_L}\int_{\mathbb T_L}
 \frac{|f(\theta)-f(\eta)|^r}
 {d_L(\theta,\eta)^{1+\alpha r}}\dd\theta\dd\eta.
$$
We set $\|f\|_{W^{\alpha,r}}:=\|f\|_{L^r}+[f]_{W^{\alpha,r}}$ and $W^{0,r}(\mathbb T_L):=L^r(\mathbb T_L)$. By differentiating \eqref{eq: 1.4} in distributions, we obtain
\begin{equation}
 q_t+(uq)_\theta=0,
 \qquad
 (\partial_\theta^3+4\partial_\theta)u=2q.
 \label{eq: 1.6}
\end{equation}

Let $\quot:\mathbb R\to\mathbb T_L$ be the quotient map, and let $\chi(t,\cdot):\mathbb T_L\to\mathbb T_L$ be the Lagrangian flow generated by $u$:
\begin{equation*}
    \partial_t\chi(t,\alpha)
 =u(t,\chi(t,\alpha)),
 \qquad
 \chi(0,\alpha)=\alpha\quad\text{for}\quad \alpha\in\mathbb T_L.
\end{equation*}
Identify $u$ with its $L$-periodic lift to $\mathbb R$, and let $\widehat\chi(t,\cdot)$ be the degree-one lift of $\chi(t,\cdot)$ determined by
\begin{equation}
 \partial_t\widehat\chi(t,\alpha)
 =u(t,\widehat\chi(t,\alpha)),
 \qquad
 \widehat\chi(0,\alpha)=\alpha,
 \qquad
 \widehat\chi(t,\alpha+L)=\widehat\chi(t,\alpha)+L.
 \label{eq: 1.7}
\end{equation}
Thus
$$
 \chi(t,\quot(\alpha))
 =\quot\bigl(\widehat\chi(t,\alpha)\bigr),
 \qquad \alpha\in\mathbb R,
$$
and $\chi(t,\cdot)$ is an orientation-preserving bi-Lipschitz homeomorphism of $\mathbb T_L$.  The transport and differentiated-transport identities are
\begin{equation}
 g(t)\circ\chi(t)=g_0,
 \qquad
 Dg(t)=\chi(t)_\#Dg_0.
 \label{eq: 1.8}
\end{equation}
If $g_0\in W^{1,1}(\mathbb T_L)$ and $J(t,\alpha)=\partial_\alpha\widehat\chi(t,\alpha)$, the measure identity in \eqref{eq: 1.8} can be reformulated as
\begin{equation}
 q(t,\widehat\chi(t,\alpha))J(t,\alpha)=q_0(\alpha).
 \label{eq: 1.9}
\end{equation}
Global uniqueness follows from \cite[Proposition~3.5]{EJ}. Although the cited result is stated forward in time, applying the involution $g(t,\theta)\mapsto-g(-t,\theta)$ to the forward solution with initial data $-g_0$ produces the negative-time branch.  Uniqueness then yields a complete solution.

\subsection{Main results}
Our first result concerns the growth of the angular derivative in both time directions.
\begin{thm}\label{thm: 1.1}
Let $p\in(1,\infty]$ and assume that
\begin{equation}
 g_0\in W^{1,p}(\mathbb T_L),
 \qquad
 g_0\ \text{is not a constant}.
 \label{eq: 1.10}
\end{equation}
Then the corresponding unique complete $W^{1,p}(\mathbb T_L)$ solution of \eqref{eq: 1.4} satisfies
\begin{equation}
 \lim_{t\to+\infty}\|g_\theta(t)\|_{L^p}=+\infty,
 \qquad
 \lim_{t\to-\infty}\|g_\theta(t)\|_{L^p}=+\infty.
 \label{eq: 1.11}
\end{equation}
In particular, if $g_0\in W^{1,\infty}(\mathbb T_L)$ is nonconstant, then \eqref{eq: 1.11} holds for every fixed $p\in(1,\infty]$.

At the endpoint, one has the following Orlicz refinement.  Set $\Phi(\xi):=\xi\log(e+\xi)$ for $\xi\geq0$.  If $g_0\in W^{1,1}(\mathbb T_L)$ is nonconstant and $D_{\theta}g_0$ satisfies
\begin{equation}
 \int_{\mathbb T_L}\Phi(|D_{\theta}g_0|)\dd\theta<\infty,
 \label{eq: 1.12}
\end{equation}
then the $L\log L$ modular of the angular derivative diverges in both time directions:
\begin{equation}
 \int_{\mathbb T_L}\Phi(|g_\theta(t)|)\dd\theta\longrightarrow\infty
 \qquad(t\to\pm\infty).
 \label{eq: 1.13}
\end{equation}
\end{thm}

For absolutely continuous data, \eqref{eq: 1.9} and a change of variables show that $\|g_\theta(t)\|_{L^1}=\|D_{\theta}g_0\|_{L^1}$ for all $t\in\mathbb R$, so the assumption that $p>1$ in the Lebesgue-scale statement is sharp.

\begin{cor}
\label{cor: 1.2}
Let $1\leq p\leq\infty$, let $g$ be the complete solution of \eqref{eq: 1.4} with $g(0)=g_0\in W^{1,p}(\mathbb T_L)$, and let $\omega(t,r,\theta)=g(t,\theta)$ be its zero-homogeneous Euler lift.  Then for every $0<a<b<\infty$ and $p<\infty$,
\begin{equation}
 \|\nabla\omega(t)\|_{L^p(\{a<|x|<b\})}^p
 =3\left(\int_a^b r^{1-p}\dd r\right)
   \|g_\theta(t)\|_{L^p(\mathbb T_L)}^p,
 \label{eq: 1.14}
\end{equation}
while
\begin{equation}
 \|\nabla\omega(t)\|_{L^\infty(\{a<|x|<b\})}
 =\frac1a\|g_\theta(t)\|_{L^\infty(\mathbb T_L)}.
 \label{eq: 1.15}
\end{equation}
Thus, for nonconstant $g_0\in W^{1,p}(\mathbb T_L)$ and $1<p\leq\infty$, the corresponding annular norm tends to $+\infty$ as $t\to\pm\infty$. When $p=1$, it is instead conserved and equals $3(b-a)\|D_{\theta}g_0\|_{L^1(\mathbb T_L)}$.
\end{cor}

\begin{proof}
In polar coordinates, $\nabla\omega=r^{-1}g_\theta\,e_\theta$.  Integrating in $r$ and using the three copies of $\mathbb T_L$ in the full angular circle, we obtain \eqref{eq: 1.14}--\eqref{eq: 1.15}. The remaining assertions follow from \cref{thm: 1.1} and conservation of $\|g_\theta(t)\|_{L^1}$.
\end{proof}

We will use $\delta_b$ to denote the unit Dirac mass at $b$ and $\mu\llcorner I$ to denote the restriction of $\mu$ to $I$.

\begin{defn}\label{def: 1.1}
We say that $g_0\in BV(\mathbb T_L)$ admits a finite decomposition if there exist nonnegative integers $M_+$, $M_-$, and $K$, families
$$
 \{(\mu_j^+,I_j^+)\}_{j=1}^{M_+},\qquad
 \{(\mu_j^-,I_j^-)\}_{j=1}^{M_-},\qquad
 \{(c_k,b_k)\}_{k=1}^{K},
$$
and a representation
\begin{equation}
 Dg_0=
 \sum_{j=1}^{M_+}\mu_j^+
 -\sum_{j=1}^{M_-}\mu_j^-
 +\sum_{k=1}^{K}c_k\delta_{b_k},
 \label{eq: 1.16}
\end{equation}
where the following conditions hold:
\begin{enumerate}[label=(\alph*)]
\item each $\mu_j^\pm$ is a nonzero, non-atomic, finite positive measure.
\item each $I_j^\pm$ is a nondegenerate oriented open arc, and all the arcs in the two families $\{I_j^+\}_{j=1}^{M_+}$ and $\{I_j^-\}_{j=1}^{M_-}$ are pairwise disjoint.  The measure $\mu_j^\pm$ is concentrated on $I_j^\pm$.  We fix a lift
$$
 I_j^\pm=\quot((a,b)),
 \qquad a<b\leq a+L.
$$
The endpoint case $b=a+L$ represents the circle with one cut point removed.
\item the restrictions of $Dg_0$ to these arcs satisfy
$$
 Dg_0\llcorner I_j^+=\mu_j^+,
 \qquad
 Dg_0\llcorner I_j^-=-\mu_j^-.
$$
\item $c_k\neq0$, the points $b_k$ are distinct and lie outside all the open arcs $I_j^\pm$, and atoms at the same point have been combined.
\end{enumerate}
Here ``concentrated on'' means that the complement of the indicated open arc has zero measure.  This convention does not require the topological support to avoid the endpoints.
\end{defn}

Define
\begin{equation}
 m_j^\pm:=\mu_j^\pm(\mathbb T_L),
 \qquad
 N_+:=K+M_++2M_-,
 \qquad
 N_-:=K+2M_++M_-.
 \label{eq: 1.17}
\end{equation}
The conclusions below hold for every fixed decomposition \eqref{eq: 1.16}.  When several decompositions are available, the forward and backward bounds may be optimized separately by minimizing $N_+$ and $N_-$, respectively.

For $B\geq0$ and an integer $N\geq0$, define
\begin{equation}
 \mathcal P_N(B):=
 \left\{f\in L^\infty(\mathbb T_L):
 \|f\|_\infty\leq B,\quad
 Df\in\mathcal M(\mathbb T_L),\quad
 \#\supp Df\leq N\right\}.
 \label{eq: 1.18}
\end{equation}
Thus $\mathcal P_N(B)$ consists of bounded step functions with at most $N$ jumps, together with the constant functions in $\mathcal P_0(B)$.  Under the zero-homogeneous lift, such a profile is piecewise constant on angular sectors separated by finitely many rays.  Let
\begin{equation}
 \nu_j^\pm(t):=\chi(t)_\#\mu_j^\pm.
 \label{eq: 1.19}
\end{equation}
For $m>0$ and an integer $\ell\geq1$, we define the following set of measures
\begin{equation}
 \mathfrak A_\ell(m):=
 \{\nu\geq0:\nu(\mathbb T_L)=m,\ \#\supp\nu\leq \ell\}.
\end{equation}
We use the bounded-Lipschitz distance
\begin{equation}
 d_{\mathrm{BL}}(\mu,\nu):=
 \sup_{\substack{\varphi\in C^{0,1}(\mathbb T_L)\\
                  \|\varphi\|_\infty+\Lip(\varphi)\leq1}}
 \left|\int_{\mathbb T_L}\varphi\dd(\mu-\nu)\right|,
 \label{eq: 1.21}
\end{equation}
where $\Lip(\varphi)$ is computed with the geodesic distance on $\mathbb T_L$.  For a set $\mathcal A$ of measures, define
$$
 d_{\mathrm{BL}}(\mu,\mathcal A)
 :=\inf_{\nu\in\mathcal A}d_{\mathrm{BL}}(\mu,\nu).
$$
On the compact circle, $d_{\mathrm{BL}}$ metrizes weak-* convergence on each set of positive measures with fixed total mass.

More generally, if $X$ is a normed space, $f\in X$, and $\mathcal E\subset X$, we write
$$
 \dist_X(f,\mathcal E):=\inf_{h\in\mathcal E}\|f-h\|_X.
$$

\begin{thm}
\label{thm: 1.3}
Let $g_0\in BV(\mathbb T_L)$ admit \eqref{eq: 1.16}, and let $g$ be the corresponding complete solution of \eqref{eq: 1.4}.  Then each assertion involving $\nu_j^+$ holds for $1\leq j\leq M_+$, and each assertion involving $\nu_j^-$ holds for $1\leq j\leq M_-$:
\begin{align}
 d_{\mathrm{BL}}\bigl(\nu_j^+(t),\mathfrak A_1(m_j^+)\bigr)&\longrightarrow0,
 &
 d_{\mathrm{BL}}\bigl(\nu_j^-(t),\mathfrak A_2(m_j^-)\bigr)&\longrightarrow0
 &&(t\to+\infty),
 \label{eq: 1.22}\\
 d_{\mathrm{BL}}\bigl(\nu_j^-(t),\mathfrak A_1(m_j^-)\bigr)&\longrightarrow0,
 &
 d_{\mathrm{BL}}\bigl(\nu_j^+(t),\mathfrak A_2(m_j^+)\bigr)&\longrightarrow0
 &&(t\to-\infty).
 \label{eq: 1.23}
\end{align}
Equivalently, every sequence tending to $+\infty$ has a common subsequence on which all positive components converge to single atoms of the correct masses and all negative components converge to measures supported on at most two points.  The signs exchange roles at $-\infty$.

Moreover, for every $1\leq r<\infty$ and $0\leq\alpha<1/r$,
\begin{align}
 \dist_{W^{\alpha,r}}\bigl(g(t),\mathcal P_{N_+}(\|g_0\|_\infty)\bigr)
 &\longrightarrow0 &&(t\to+\infty),
 \label{eq: 1.24}\\
 \dist_{W^{\alpha,r}}\bigl(g(t),\mathcal P_{N_-}(\|g_0\|_\infty)\bigr)
 &\longrightarrow0 &&(t\to-\infty).
 \label{eq: 1.25}
\end{align}
Each set $\mathcal P_N(B)$ is invariant under the flow and compact in $W^{\alpha,r}(\mathbb T_L)$ throughout this range.  Thus \eqref{eq: 1.24}--\eqref{eq: 1.25} hold along the entire positive and negative half-orbits, not only along subsequences.
\end{thm}

Let $J_+$, $J_-$, and $J_0$ be at most countable index sets.  We assume that $g_0\in BV(\mathbb T_L)$ has the decomposition
\begin{equation}
 Dg_0=\sum_{j\in J_+}\mu_j^+
      -\sum_{j\in J_-}\mu_j^-
      +\sum_{k\in J_0}c_k\delta_{b_k}
 \quad\text{in }\mathcal M(\mathbb T_L),
 \label{eq: 1.26}
\end{equation}
where the measures, arcs, and atoms satisfy the corresponding conditions in \cref{def: 1.1}, with the finite families there replaced by the families indexed by $J_+$, $J_-$, and $J_0$.  For $j\in J_\pm$, set $m_j^\pm:=\mu_j^\pm(\mathbb T_L)$.  We additionally require
\begin{equation}
 \sum_{j\in J_+}m_j^+
 +\sum_{j\in J_-}m_j^-
 +\sum_{k\in J_0}|c_k|<\infty.
 \label{eq: 1.27}
\end{equation}
The series in \eqref{eq: 1.26} then converges absolutely in total variation.  We retain the notation $\nu_j^\pm(t):=\chi(t)_\#\mu_j^\pm$. We call $h\in BV(\mathbb T_L)$ a \emph{jump profile} if $Dh$ is purely atomic.  The support of $Dh$ need not be finite.

For the complete orbit through $g_0$, we define the one-sided limit sets
\begin{align}
 \Omega_+(g_0)
 &:=\bigcap_{T>0}\overline{\{g(t):t\geq T\}}^{\,L^1},
 \label{eq: 1.28}\\
 \Omega_-(g_0)
 &:=\bigcap_{T>0}\overline{\{g(t):t\leq-T\}}^{\,L^1}.
 \label{eq: 1.29}
\end{align}

\begin{thm}
\label{thm: 1.4}
Let $g_0\in BV(\mathbb T_L)$ satisfy \eqref{eq: 1.26}--\eqref{eq: 1.27}.  Then for every fixed component, the conclusions \eqref{eq: 1.22}--\eqref{eq: 1.23} hold with the same one- and two-atom bounds.  Moreover, the limit sets $\Omega_\pm(g_0)$ are nonempty, flow-invariant, and compact in $W^{\alpha,r}(\mathbb T_L)$ for every $1\leq r<\infty$ and $0\leq\alpha<1/r$.  They attract the corresponding half-orbits:
\begin{align}\label{eq: 1.30}
	\begin{cases}
		\dist_{W^{\alpha,r}}(g(t),\Omega_+(g_0))\longrightarrow0
		\ \ &\quad t\to+\infty,\\
		\dist_{W^{\alpha,r}}(g(t),\Omega_-(g_0))\longrightarrow0 \ \ & \quad t\to-\infty.
	\end{cases}
\end{align}
For every $h\in\Omega_+(g_0)$ there are points $x_j^+,y_k^+\in\mathbb T_L$, depending on $h$, and measures $\rho_j^-\in\mathfrak A_2(m_j^-)$ such that
\begin{equation}
 Dh=\sum_{j\in J_+}m_j^+\delta_{x_j^+}
    -\sum_{j\in J_-}\rho_j^-
    +\sum_{k\in J_0}c_k\delta_{y_k^+},
 \label{eq: 1.31}
\end{equation}
while for every $h\in\Omega_-(g_0)$ there are points $x_j^-,y_k^-\in\mathbb T_L$, depending on $h$, and measures $\rho_j^+\in\mathfrak A_2(m_j^+)$ such that
\begin{equation}
 Dh=\sum_{j\in J_+}\rho_j^+
    -\sum_{j\in J_-}m_j^-\delta_{x_j^-}
    +\sum_{k\in J_0}c_k\delta_{y_k^-}.
 \label{eq: 1.32}
\end{equation}
The series in \eqref{eq: 1.31}--\eqref{eq: 1.32} converge absolutely in total variation.  In particular, the limiting derivatives have neither an absolutely continuous nor a Cantor part.
\end{thm}

Every profile in the attracting omega-limit sets of \cref{thm: 1.4} has purely atomic derivative.

\begin{cor}
\label{cor: 1.5}
Let $g_0\in C^1(\mathbb T_L)$ be nonconstant.  Then, for every $p\in(1,\infty]$, we have $\|g_\theta(t)\|_{L^p}\to\infty$ as $t\to\pm\infty$, and each half-orbit is attracted to the corresponding compact limit set of jump profiles in \cref{thm: 1.4}.  If $|t_n|\to\infty$ and $g(t_n)\rightharpoonup h$ weakly in $L^2(\mathbb T_L)$, then $g(t_n)\to h$ strongly in $W^{\alpha,r}(\mathbb T_L)$ for every $1\leq r<\infty$, $0\leq\alpha<1/r$.  Moreover, $Dh$ is purely atomic and the complete orbit through $h$ is precompact in all these spaces.
\end{cor}

\begin{cor}
\label{cor: 1.6}
Let $\omega_0\in L^1(\mathbb R^2)\cap L^\infty(\mathbb R^2)$ be compactly supported, 3-fold symmetric, smooth on $\mathbb R^2\setminus\{0\}$, and satisfy
\begin{equation}
 \int_{\mathbb R^2}\omega_0(x)\dd x=0.
 \label{eq: 1.33}
\end{equation}
Suppose that, for a nonconstant $g_0\in C^1(\mathbb T_L)$,
\begin{equation}
 \omega_0(r,\theta)\longrightarrow g_0(\theta)
 \quad\text{for almost every $\theta$ as $r\to0$}.
 \label{eq: 1.34}
\end{equation}
Let $\omega$ be the corresponding complete Euler solution.  For every finite time, $\omega(t)$ is smooth on $\mathbb R^2\setminus\{0\}$.  Its velocity has finite kinetic energy, and its radial trace $g(t)$ solves \eqref{eq: 1.4}.  Then for every $1<p\leq\infty$,
\begin{equation}
 \lim_{t\to\pm\infty}\ \liminf_{r\to0}
 \|r\nabla\omega(t,r,\cdot)\|_{L^p(\mathbb T_L)}=+\infty.
 \label{eq: 1.35}
\end{equation}
The two half-orbits of radial traces are attracted to the compact sets of jump profiles in \cref{cor: 1.5}.  In particular, every infinite-time $L^1$ limit of the traces has purely atomic derivative.
\end{cor}

Before giving its proof, we point out that the nonconstant trace in \eqref{eq: 1.34} makes $\omega_0$ necessarily discontinuous at the origin.  Accordingly, the smoothness hypothesis is imposed on $\mathbb R^2\setminus\{0\}$.

\begin{proof}
Transport preserves the circulation and leaves the support compact at each finite time.  Denote the planar velocity by $U$.  For large $|x|$,
$$
 U(t,x)=\frac1{2\pi}\int_{\mathbb R^2}
 \left(\frac{(x-y)^\perp}{|x-y|^2}-\frac{x^\perp}{|x|^2}\right)
 \omega(t,y)\dd y=O(|x|^{-2}).
$$
Together with $\omega(t)\in L^1\cap L^\infty$, this implies $U(t,\cdot)\in L^2(\mathbb R^2)$ for every $t$.

3-fold symmetry fixes the origin.  The velocity estimate \cite[(1.8)]{EJ} and $\|\omega(t)\|_{L^\infty}=\|\omega_0\|_{L^\infty}$ imply that every particle path $X(t)$ satisfies
$$
 e^{-C\|\omega_0\|_\infty|t|}|X(0)|
 \leq |X(t)|
 \leq e^{C\|\omega_0\|_\infty|t|}|X(0)|.
$$
Thus a particle starting away from the origin stays away from the origin on every finite time interval.  Since $\omega_0$ is smooth on $\mathbb R^2\setminus\{0\}$, local propagation of regularity as in \cite[Section~2.3]{EJ} gives
$$
 \omega(t)\in C^\infty_{\mathrm{loc}}(\mathbb R^2\setminus\{0\})
 \qquad(t\in\mathbb R).
$$

By \cite[Proposition~1.8]{EMS}, the radial trace exists for every $t$ and is the solution of \eqref{eq: 1.4} with initial value $g_0$.  The uniform vorticity bound and dominated convergence imply that as $r\to 0$,
$$
 \omega(t,r,\cdot)\longrightarrow g(t)
 \quad\text{in }L^p(\mathbb T_L),\qquad 1\leq p<\infty.
$$
Weak lower semicontinuity in $W^{1,p}(\mathbb {T}_L)$, and weak-* lower semicontinuity in $W^{1,\infty}(\mathbb {T}_L)$, therefore yield the following estimate.  For $p=\infty$, one first takes a sequence realizing the lower limit.  If that lower limit is finite, weak-* compactness and the finite-$p$ convergence identify the limit as $g(t)$.  Thus
$$
 \|g_\theta(t)\|_{L^p}
 \leq\liminf_{r\to0}
 \|\partial_\theta\omega(t,r,\cdot)\|_{L^p}
 \leq\liminf_{r\to0}
 \|r\nabla\omega(t,r,\cdot)\|_{L^p}.
$$
Now \cref{thm: 1.1} proves \eqref{eq: 1.35}.  By \cref{cor: 1.5}, we obtain the assertions about the trace limit sets.
\end{proof}

\begin{cor}\label{cor: 1.7}
Let $n\geq1$ be an integer, and let $g_0\in C^\infty(\mathbb T_L)$ be a nonconstant Morse function with exactly $n$ local maxima and $n$ local minima, so $g_{0,\theta}(\theta)=0$ implies $g_{0,\theta\theta}(\theta)\neq0$.  Then, for every $p\in(1,\infty]$,
$$
 \|g_\theta(t)\|_{L^p(\mathbb T_L)}\longrightarrow+\infty
 \qquad t\to\pm\infty.
$$
Moreover, for every $1\leq r<\infty$ and $0\leq\alpha<1/r$,
\begin{equation}
 \dist_{W^{\alpha,r}}\bigl(g(t),\mathcal P_{3n}(\|g_0\|_\infty)\bigr)
 \longrightarrow0
 \qquad t\to\pm\infty.
 \label{eq: 1.36}
\end{equation}
Moreover, if $t_k\to+\infty$ or $t_k\to-\infty$ and $g(t_k)\rightharpoonup h$ weakly in $L^2$, then $h\in\mathcal P_{3n}(\|g_0\|_\infty)$, and the complete orbit through $h$ is precompact in $W^{\alpha,r}(\mathbb T_L)$ throughout the stated range.
\end{cor}

By the standard genericity theorem for Morse functions \cite[Chapter~6, \S1]{Hirsch}, the set of nonconstant Morse functions is open and dense in $C^\infty(\mathbb T_L)$ with its usual Fr\'{e}chet topology. For each fixed $n$, the subclass with exactly $n$ local maxima and $n$ local minima is open; the density assertion concerns the union of these subclasses. Hence \cref{cor: 1.7} applies to an open dense class of nonconstant smooth profiles.  The count $3n$ comes from the alternating increasing and decreasing arcs, see \cref{subsec: 3.4}.

We quotient by rigid rotations, identifying $f$ with $f(\cdot+\beta)$. When a topology on this quotient is needed, we use the quotient metric induced by $L^s(\mathbb T_L)$.

An orbit is a travelling wave if $g(t,\theta)=h(\theta-ct)$ for some profile $h$ and constant $c$.  We say that it is heteroclinic modulo rotations if its quotient orbit converges, as $t\to-\infty$ and $t\to+\infty$, to the classes of travelling waves with fewer jumps. A limit of time translates is a complete solution $\bar g$ obtained as $g(t_n+\cdot)\to\bar g$ in $C_{\mathrm{loc}}(\mathbb R;L^s)$, for some (equivalently, every) $1\leq s<\infty$, along a sequence $t_n\to+\infty$ or $t_n\to-\infty$.

\begin{thm}
\label{thm: 1.8}
Fix $B\geq0$, let $g_*\in\mathcal P_3(B)$, and let $\bar g$ be the complete $BV$ solution of \eqref{eq: 1.4} with $\bar g(0)=g_*$.  Write $g_*$ in its minimal jump representation, with zero jumps removed and adjacent arcs carrying the same value merged. Then exactly one of the following occurs:
\begin{enumerate}[label=(\roman*)]
\item $g_*$ is constant, and the orbit is fixed.
\item $g_*$ has two nonzero jumps, and its orbit is a travelling wave.
\item $g_*$ has three nonzero jumps.  Label the jump points cyclically by $a_0<a_1<a_2<a_0+L$, set $\Delta_i:=g_*(a_i+)-g_*(a_i-)$, and define
$$
 x_0=a_1-a_0,\qquad y_0=a_2-a_1,\qquad z_0=L-x_0-y_0.
$$
There exist a bounded open interval $(\sigma_-,\sigma_+)$ containing $0$ and a strictly increasing function $\sigma:\mathbb R\to(\sigma_-,\sigma_+)$ such that
$$
 (x(t),y(t),z(t))
 =(x_0,y_0,z_0)+\sigma(t)(\Delta_2,\Delta_0,\Delta_1),
$$
where $x(t),y(t),z(t)$ are the three consecutive interval lengths.  Moreover,
$$
 \sigma(t)\to\sigma_-\quad(t\to-\infty),
 \qquad
 \sigma(t)\to\sigma_+\quad(t\to+\infty).
$$
The endpoint profiles have fewer jumps: an endpoint in the interior of an edge gives a two-jump travelling wave, while a vertex gives a constant profile.  Thus the orbit is heteroclinic modulo rotations, and no genuine three-jump orbit is periodic.
\end{enumerate}
If one interval length vanishes at an endpoint, then the limiting profile has two jumps.  If two interval lengths vanish, then it is constant.

Finally, let $g_0\in BV(\mathbb T_L)$ admit \eqref{eq: 1.16} with $K=0$ and $M_+=M_-=1$.  The orbit of every complete solution obtained as a limit of time translates of the solution through $g_0$, along a sequence tending to either $+\infty$ or $-\infty$, is one of the three classes listed in (i)--(iii).
\end{thm}

\subsection{Key ideas of the proof}

For $m\geq4$, the relaxation result in \cite{EMS} follows from the positivity of the symmetrized kernel.  This property fails when $m=3$, since
$$
 K^3_{1/(\partial_{\theta\theta}+4)}(\theta)
 =\frac{3\pi}{10}\left|\sin\frac{3\theta}{2}\right|-\frac7{20}.
$$
We recover the missing sign by finding key monotonic quantities in two different ways.

For gradient growth, let $q=g_\theta$ and
$$
 \mathcal H(t):=\int_{\mathbb T_L}q\log|q|\dd\theta,
 \qquad
 \mathcal H'(t)
 =2\int_{\mathbb T_L}(G_{\theta\theta}^2-4G_\theta^2)\dd\theta
 \geq10\|G_\theta(t)\|_2^2.
$$
Suppose that $\|q(t_n)\|_{L^p}$ stays bounded for some $t_n\to+\infty$.  The corresponding bound for $\mathcal H(t_n)$ and the monotonicity of $\mathcal H$ give
$$
 \int_0^\infty\|G_\theta(t)\|_2^2\dd t<\infty.
$$
On the other hand, the bound for $q(t_n)$ propagates to a fixed time interval around each $t_n$.  The functions on these time intervals have compact closure in $C^0$, while
$$
 \osc g(t)=\osc g_0>0.
$$
Since $G_\theta=0$ only when $g$ is constant, compactness gives
$$
 \|G_\theta(t_n+\sigma)\|_2\geq c,
 \qquad |\sigma|\leq\delta,
$$
for some $c,\delta>0$ independent of $n$.  Taking a subsequence for which these time intervals are disjoint contradicts the preceding time integrability.  The same proof applied to $-g(-t)$ gives the result for negative time.

For relaxation, positive and negative components of $D_\theta g$ are transported separately.  In the projective coordinate $z=\tan(\theta-c)$, write the coordinates of four ordered particles as $x_1<x_2<x_3<x_4$ and set
$$
 U(t,z):=\frac{u(t,\theta)}{\cos^2(\theta-c)},
 \qquad
 R:=\frac{(x_2-x_1)(x_4-x_3)}
          {(x_3-x_1)(x_4-x_2)}.
$$
Direct differentiation gives
$$
\begin{aligned}
 \frac{\dd}{\dd t}\log R=L_xU,
 \qquad
 L_xU
 &:=\frac{U(x_2)-U(x_1)}{x_2-x_1}
 +\frac{U(x_4)-U(x_3)}{x_4-x_3}\\
 &\quad
 -\frac{U(x_3)-U(x_1)}{x_3-x_1}
 -\frac{U(x_4)-U(x_2)}{x_4-x_2}.
\end{aligned}
$$
Moreover,
$$
 D_z^3U=z_\#\bigl(2\cos^2(\theta-c)D_\theta g\bigr),
 \qquad
 L_xU=\int_{[x_1,x_4]}P_x(s)\dd(D_z^3U)(s),
$$
where $P_x$ is the third-order Peano kernel associated with $L_x$ and $P_x>0$ on $(x_1,x_4)$.  The identities $L_x(1)=L_x(z)=L_x(z^2)=0$ explain why only $D_z^3U$ appears.  Since $P_x>0$ in the open interval, the preceding formula gives
$$
 \frac{\dd}{\dd t}\log R\geq0
$$
on every positive transported component.  On a negative component, the same formula gives
$$
 \frac{\dd}{\dd t}\log(1-R)\geq0.
$$

These two monotonicities also determine the possible limiting atoms.  For a positive component, two different limiting atoms would allow us to choose four material points such that, along the limiting sequence,
$$
 (x_1,x_2,x_3,x_4)\longrightarrow(a,a,b,b).
$$
Then $R\to0$, which contradicts $R(t)\geq R(0)>0$.  For a negative component, three different limiting atoms give instead
$$
 (x_1,x_2,x_3,x_4)\longrightarrow(a,b,b,d),
$$
and hence $R\to1$.  This contradicts $1-R(t)\geq1-R(0)>0$.  Therefore every forward limiting positive component is one atom and every forward limiting negative component has at most two atoms. For countably many components, their masses are summable.  After taking a diagonal subsequence, the derivative of every limiting profile is therefore a convergent sum of the limiting atomic measures above together with the transported initial atoms.  Hence the limiting derivative is purely atomic. In particular, neither an absolutely continuous part nor a Cantor part can remain.  This replaces the kernel positivity used for $m\geq4$.

\bigskip

The paper is organized as follows.  \Cref{sec: 2} proves gradient growth and compactness of time translates.  The projective monotonicity and limits of the transported measures leading to atomicity and relaxation to jump profiles are developed in \cref{sec: 3}, where Theorem \ref{thm: 1.3}, Theorem \ref{thm: 1.4}, Corollary \ref{cor: 1.5} and Corollary \ref{cor: 1.7} are proved.  The equations for finitely many jumps and the three-jump classification (Theorem \ref{thm: 1.8}) are proved in \cref{sec: 4}.

\section{Gradient growth and compactness}
\label{sec: 2}

In this section, we will prove Theorem \ref{thm: 1.1}. Subsections 2.1 and 2.2 contain the preliminary results and the proof, respectively. Moreover, Subsection 2.3 provides the compactness and flow-continuity tools needed for the analysis of time-translation limits and omega-limit sets in the subsequent sections.

\subsection{Flow estimates and profiles with finitely many jumps}

Since the Fourier multiplier of $A=\partial_{\theta\theta}+4$ never vanishes on $\mathbb T_L$, the operator has a periodic Green kernel $K_L$.  In one dimension, $K_L$ is continuous, and thus its first derivative is bounded with one jump per period, and
\begin{equation}
 A^{-1}f=K_L*f.
 \label{eq: 2.1}
\end{equation}
Using Young's inequality, first with $K_L$ and then with its derivative $K_L'$, yields
\begin{align}
 \|A^{-1}f\|_{L^\infty}
 &\leq \|K_L\|_{L^1}\|f\|_{L^\infty},
 \label{eq: 2.2}\\
 \|\partial_\theta A^{-1}f\|_{L^\infty}
 &\leq \|K_L'\|_{L^1}\|f\|_{L^\infty}.
 \label{eq: 2.3}
\end{align}
Consequently, with a constant depending only on $L$,
\begin{equation}
 \|u\|_{W^{1,\infty}}
 =2\|A^{-1}g\|_{W^{1,\infty}}
 \leq C_L\|g\|_{L^\infty}.
 \label{eq: 2.4}
\end{equation}
The same kernel satisfies the smoothing estimate used in the compactness arguments.  For $h\in L^1$,
\begin{equation}
 \|2A^{-1}h\|_{C^1}
 \leq 2\bigl(\|K_L\|_{L^\infty}
              +\|K_L'\|_{L^\infty}\bigr)\|h\|_{L^1}
 \leq C_L\|h\|_{L^1}.
 \label{eq: 2.5}
\end{equation}
Although $K_L'$ has one jump per period, $K_L'*h$ is continuous.  In fact, translation continuity in $L^1$ shows that
$$
 \|(K_L'*h)(\cdot+\eta)-(K_L'*h)(\cdot)\|_\infty
 \leq \|K_L'\|_\infty
       \|h(\cdot+\eta)-h(\cdot)\|_1\longrightarrow0.
$$
Since $\mathbb T_L$ has finite measure, every bounded solution belongs to $L^1$.  Taking $h=g(t)$ in \eqref{eq: 2.5} shows that $u(t)=2A^{-1}g(t)\in C^1(\mathbb T_L)$ at each fixed time.  This $C^1$ regularity and the uniform estimate \eqref{eq: 2.4} justify the classical flow, its spatial derivative, and the pointwise Jacobian equation used below, without requiring any derivative of $g$.

On finite time intervals, $u\in L^\infty_tC^1_\theta$.  By classical characteristic theory, $g(t)=g_0\circ\chi(t)^{-1}$.  The Sobolev and $BV$ chain rules for the $C^1$ bi-Lipschitz flow then imply, respectively, \eqref{eq: 1.9} and $Dg(t)=\chi(t)_\#Dg_0$ at the stated regularity.

Writing $\osc f:=\operatorname*{ess\,sup}f-\operatorname*{ess\,inf}f$, the transport identity preserves the essential range, and hence
\begin{equation}
 \|g(t)\|_\infty=\|g_0\|_\infty,
 \qquad
 \osc g(t)=\osc g_0.
 \label{eq: 2.6}
\end{equation}
Set
\begin{equation}
 \Lambda:=C_L\|g_0\|_\infty,
 \label{eq: 2.7}
\end{equation}
where $C_L$ is the same as that in \eqref{eq: 2.4}. For $s\in\mathbb R$, let $\widehat\chi(t;s,\alpha)$ be the degree-one lift of the flow starting at time $s$, defined by
$$
 \partial_t\widehat\chi(t;s,\alpha)
 =u(t,\widehat\chi(t;s,\alpha)),
 \qquad
 \widehat\chi(s;s,\alpha)=\alpha,
 \qquad
 \widehat\chi(t;s,\alpha+L)=\widehat\chi(t;s,\alpha)+L.
$$
Set $J(t;s,\alpha):=\partial_\alpha\widehat\chi(t;s,\alpha)$.  Then
\begin{equation}
 \partial_tJ(t;s,\alpha)
 =u_\theta(t,\widehat\chi(t;s,\alpha))J(t;s,\alpha),
 \qquad J(s;s,\alpha)=1.
 \label{eq: 2.8}
\end{equation}
Using \eqref{eq: 2.4} in \eqref{eq: 2.8} and integrating, we obtain
\begin{equation}
 e^{-\Lambda|t-s|}
 \leq J(t;s,\alpha)
 \leq e^{\Lambda|t-s|}.
 \label{eq: 2.9}
\end{equation}
In particular, the flow is order-preserving and bi-Lipschitz for every finite time.

The Jacobian bound also implies the finite-time $W^{1,p}$ estimates that will be used later.  The relative-time version of \eqref{eq: 1.9} is
$$
 q(t,\widehat\chi(t;s,\alpha))J(t;s,\alpha)=q(s,\alpha).
$$
Using this identity and the change of variables $\theta=\widehat\chi(t;s,\alpha)$ modulo $L$, we obtain, for $1<p<\infty$,
\begin{align}
 \|q(t)\|_p^p
 &=\int_{\mathbb T_L}|q(t,\theta)|^p\dd\theta
 \notag
 =\int_{\mathbb T_L}
   \left|\frac{q(s,\alpha)}{J(t;s,\alpha)}\right|^p
   J(t;s,\alpha)\dd\alpha
 \notag\\
 &=\int_{\mathbb T_L}|q(s,\alpha)|^p
   J(t;s,\alpha)^{1-p}\dd\alpha
 \notag
 \leq e^{(p-1)\Lambda|t-s|}\|q(s)\|_p^p,
 \label{eq: 2.10}
\end{align}
which leads to
\begin{equation}
 \|q(t)\|_p
 \leq e^{(1-1/p)\Lambda|t-s|}\|q(s)\|_p.
 \label{eq: 2.11}
\end{equation}
For $p=+\infty$, \eqref{eq: 1.9} and the lower bound in \eqref{eq: 2.9} imply
\begin{equation}
 \|q(t)\|_\infty
 \leq e^{\Lambda|t-s|}\|q(s)\|_\infty.
 \label{eq: 2.12}
\end{equation}
Thus $W^{1,p}$ regularity propagates on every finite time interval.

Finally, since push-forward by a homeomorphism preserves the total variation of a signed measure, \eqref{eq: 1.8} yields
\begin{equation}
 \|Dg(t)\|_{\mathcal M}
 =\|Dg_0\|_{\mathcal M}.
 \label{eq: 2.13}
\end{equation}

\begin{lem}
\label{lem: 2.1}
If $f_n\to f$ in $L^1(\mathbb T_L)$ and
$$
 \sup_n\bigl(\|f_n\|_\infty+\|Df_n\|_{\mathcal M}\bigr)<\infty,
$$
then $f_n\to f$ in $W^{\alpha,r}(\mathbb T_L)$ for every $1\leq r<\infty$ and $0\leq\alpha<1/r$.  More precisely, if $I\subset\mathbb R$ is any interval, $f_n(\tau)\to f(\tau)$ in $L^1$ uniformly for $\tau\in I$, and
$$
 \sup_{n,\,\tau\in I}
 \bigl(\|f_n(\tau)\|_\infty+\|Df_n(\tau)\|_{\mathcal M}\bigr)<\infty,
$$
then the convergence in every such $W^{\alpha,r}(\mathbb{T}_L)$ is uniform for $\tau\in I$.
\end{lem}

\begin{proof}
Lower semicontinuity shows that $f$ belongs to $BV$ with the same uniform variation bound.  After passing to an almost-everywhere convergent subsequence, the uniform $L^\infty$ bound passes to $f$.  Let $h_n=f_n-f$.  Interpolation implies $\|h_n\|_r^r\leq\|h_n\|_\infty^{r-1}\|h_n\|_1\to0$.  This proves the case $\alpha=0$.  Suppose that $0<\alpha<1/r$.  For $0<z<L/2$, the $BV$ translation estimate and the triangle inequality imply
$$
\|h_n(\cdot+z)-h_n\|_1
 \le \min\{2\|h_n\|_1,\,
 z\|Dh_n\|_{\mathcal M}\}.
$$
Together with the uniform $L^\infty$ and BV bounds, this yields
$$
\|h_n(\cdot+z)-h_n\|_r^r
 \le
 \|h_n(\cdot+z)-h_n\|_\infty^{r-1}
 \|h_n(\cdot+z)-h_n\|_1
 \lesssim \min\{\|h_n\|_1,z\}.
$$
Hence the periodic translation formula for the Gagliardo seminorm yields
$$
[h_n]_{W^{\alpha,r}}^r
 \lesssim
 \int_0^{L/2}
 \frac{\min\{\|h_n\|_1,z\}}{z^{1+\alpha r}}\dd z.
$$
Splitting at $0<\delta<L/2$, we obtain
$$
[h_n]_{W^{\alpha,r}}^r
 \lesssim
 \delta^{1-\alpha r}
 +\|h_n\|_1
 \int_\delta^{L/2} z^{-1-\alpha r}\dd z.
$$
Since $\alpha r<1$, the first term in the right-hand side of the above inequality can be made arbitrarily small by letting $\delta\to0$, while for fixed $\delta$ the second term tends to zero as $n\to\infty$.  The same estimates, with $\|h_n\|_1$ replaced by $\sup_{\tau\in I}\|h_n(\tau)\|_1$, prove the uniform version.
\end{proof}

The condition $\alpha r<1$ is sharp when the limit has a nonzero jump: the endpoint Gagliardo seminorm diverges for $r>1$, while for $r=1$ the distributional derivative has an atom and hence is not in $L^1$.

\begin{lem}\label{lem: 2.2}
For every $B\geq0$, every integer $N\geq0$, $1\leq r<\infty$, and $0\leq\alpha<1/r$, $\mathcal P_N(B)$ is compact in $W^{\alpha,r}(\mathbb T_L)$ and invariant under the flow of \eqref{eq: 1.4}.
\end{lem}

\begin{proof}
For compactness, observe that a function whose distributional derivative is supported at $k$ points is constant on each connected component of the complement of those points.  Hence it is a step function with at most $k$ jumps.  The zero-jump stratum is the compact family of constants $[-B,B]$.  Fix $1\leq k\leq N$.  After choosing a cut at $0$, represent a $k$-jump function by ordered jump positions and the constant values on the intervening arcs,
\begin{equation}
 0\leq a_1\leq\cdots\leq a_k\leq L,
 \qquad (c_1,\ldots,c_k)\in[-B,B]^k.
 \label{eq: 2.14}
\end{equation}
Take the value $c_k$ on $[0,a_1)\cup[a_k,L)$ and $c_i$ on $[a_i,a_{i+1})$, $1\leq i<k$.  Coincident $a_i$ are allowed.  They represent a stratum with fewer than $k$ jumps.  If parameters converge, the values converge uniformly away from intervals whose total length tends to zero.  More explicitly, if $f_n$ and $f$ correspond to convergent parameters, then
$$
 \|f_n-f\|_1
 \leq L\max_i|c_{i,n}-c_i|
 +C_k(2B)\sum_i|a_{i,n}-a_i|\longrightarrow0,
$$
where $C_k$ depends only on $k$. Thus the parameter map from the compact set in \eqref{eq: 2.14} to $L^1$ is continuous.  Since every member has total variation at most $2BN$, \cref{lem: 2.1} and the finite union over $0\leq k\leq N$ prove the stated compactness.

For invariance, let
$$
 Dg_0=\sum_{i=1}^k\Delta_i\delta_{a_i},
 \qquad k\leq N.
$$
Equation \eqref{eq: 1.8} implies that
$$
 Dg(t)=\sum_{i=1}^k\Delta_i\delta_{\chi(t,a_i)}.
$$
The jump sizes and the constant values are preserved.  Distinct jumps cannot collide at a finite time: on an ordered lift, \eqref{eq: 2.9} implies
\begin{equation}
 |\widehat\chi(t,\alpha)-\widehat\chi(t,\beta)|
 \geq e^{-\Lambda|t|}|\alpha-\beta|.
 \label{eq: 2.15}
\end{equation}
Therefore $g(t)\in\mathcal P_N(B)$ for every finite $t$.
\end{proof}

\subsection{Gradient growth}

Set $q=g_\theta$ and suppose that either $q_0\in L^p$ for some $p>1$ or \eqref{eq: 1.12} holds.  The Jacobian bounds propagate these conditions on every finite time interval.  Set
\begin{equation}
 \mathcal H(t):=\int_{\mathbb T_L}q(t,\theta)\log|q(t,\theta)|\dd\theta,
 \label{eq: 2.16}
\end{equation}
where the convention $0\log0:=0$ is used. The integral is finite since $|r\log|r||\leq C(1+\Phi(|r|))$, while $\Phi(|r|)\leq C_p(1+|r|^p)$ for every $p>1$.

The identity \eqref{eq: 1.9} avoids differentiating through the zero set of $q$ and leads to
$$
 q(t,\widehat\chi(t,\alpha))=\frac{q_0(\alpha)}{J(t,\alpha)}.
$$
By a change of variables $\theta=\widehat\chi(t,\alpha)$ in \eqref{eq: 2.16}, we obtain
\begin{align}
 \mathcal H(t)
 &=\int_{\mathbb T_L}
   \frac{q_0(\alpha)}{J(t,\alpha)}
   \log\left|\frac{q_0(\alpha)}{J(t,\alpha)}\right|
   J(t,\alpha)\dd\alpha
 \notag\\
 &=\int_{\mathbb T_L}q_0(\alpha)\log|q_0(\alpha)|\dd\alpha
   -\int_{\mathbb T_L}q_0(\alpha)\log J(t,\alpha)\dd\alpha
 \notag\\
 &=\mathcal H(0)
   -\int_{\mathbb T_L}q_0(\alpha)\log J(t,\alpha)\dd\alpha.
 \label{eq: 2.17}
\end{align}
The integrand in the last term is dominated on finite time intervals by $|q_0|\Lambda|t|$, so differentiation under the integral is justified.  From \eqref{eq: 2.8},
$$
 \frac{J_t(t,\alpha)}{J(t,\alpha)}
 =u_\theta(t,\widehat\chi(t,\alpha)).
$$
Using again $q_0(\alpha)\dd\alpha=q(t,\theta)\dd\theta$, we obtain
\begin{align}
 \mathcal H'(t)
 &=-\int_{\mathbb T_L}q_0(\alpha)
       \frac{J_t}{J}(t,\alpha)\dd\alpha
 =-\int_{\mathbb T_L}u_\theta(t,\theta)q(t,\theta)\dd\theta
 =-2\int_{\mathbb T_L}G_\theta g_\theta\dd\theta.
 \label{eq: 2.18}
\end{align}
Since $g=G_{\theta\theta}+4G$,
$$
 g_\theta=G_{\theta\theta\theta}+4G_\theta.
$$
At the $L\log L$ endpoint, $g\in W^{1,1}$ and elliptic regularity implies $G\in W^{3,1}$, so the following integration by parts is legitimate. Substituting into \eqref{eq: 2.18} yields
\begin{align}
 \mathcal H'(t)
 &=-2\int G_\theta G_{\theta\theta\theta}\dd\theta
   -8\int G_\theta^2\dd\theta
 \notag\\
 &=2\int G_{\theta\theta}^2\dd\theta
   -8\int G_\theta^2\dd\theta
 \notag\\
 &=2\int_{\mathbb T_L}
   \bigl(G_{\theta\theta}^2-4G_\theta^2\bigr)\dd\theta.
 \label{eq: 2.19}
\end{align}

In view of the Fourier expansion of $G$ in the basis $\{e^{i3k\theta}\}$, we can determine the sign of $\mathcal H'$.  Since $G_\theta$ has zero mean,
$$
\begin{aligned}
 \|G_{\theta\theta}\|_2^2
 &=L\sum_{k\neq0}(3k)^4|\widehat G_k|^2,\\
 \|G_\theta\|_2^2
 &=L\sum_{k\neq0}(3k)^2|\widehat G_k|^2.
\end{aligned}
$$
For $k\neq0$, $(3k)^4\geq9(3k)^2$, whence
\begin{equation}
 \|G_{\theta\theta}\|_2^2\geq9\|G_\theta\|_2^2.
 \label{eq: 2.20}
\end{equation}
Combining \eqref{eq: 2.19} and \eqref{eq: 2.20}, we get
\begin{equation}
 \mathcal H'(t)
 \geq2(9-4)\|G_\theta(t)\|_2^2
 =10\|G_\theta(t)\|_2^2\geq0,
 \label{eq: 2.21}
\end{equation}
which implies that $t\mapsto\mathcal H(t)$ is nondecreasing.

\begin{proof}[Proof of \cref{thm: 1.1}]
We argue by contradiction for positive time. Suppose first that \eqref{eq: 1.11} failed. Then there would be $t_n\to+\infty$ and $M<\infty$ such that
\begin{equation}
 \|q(t_n)\|_p\leq M.
 \label{eq: 2.22}
\end{equation}
For $1<p<\infty$, by the elementary entropy bound, we obtain
$$
 |\mathcal H(t_n)|
 \leq C_{p,L}\bigl(1+\|q(t_n)\|_p^p\bigr)
 \leq C_{p,L}(1+M^p).
$$
For $p=\infty$,
$$
 |\mathcal H(t_n)|
 \leq L\max_{|r|\leq M}|r\log|r||.
$$
Hence in both cases there is a constant $C$ such that $\mathcal H(t_n)\leq C$ for all $n$.

Fix $T>0$ and choose $n$ with $t_n\geq T$.  Monotonicity implies
$$
 \mathcal H(T)\leq\mathcal H(t_n)\leq C.
$$
Therefore $\mathcal H$ is bounded above on $[0,\infty)$.  Integrating \eqref{eq: 2.21} from $0$ to $T$ yields
$$
 10\int_0^T\|G_\theta(t)\|_2^2\dd t
 \leq\mathcal H(T)-\mathcal H(0)
 \leq C-\mathcal H(0).
$$
Letting $T\to\infty$, we have
\begin{equation}
 \int_0^\infty\|G_\theta(t)\|_2^2\dd t<\infty.
 \label{eq: 2.23}
\end{equation}

The bound \eqref{eq: 2.22} propagates to a uniform time window around each $t_n$.  Fix $\delta>0$.  The estimates \eqref{eq: 2.11}--\eqref{eq: 2.12} imply that
\begin{equation}
 \sup_{|\sigma|\leq\delta}\|q(t_n+\sigma)\|_p\leq M',
 \label{eq: 2.24}
\end{equation}
where $M'$ depends only on $M,p,\delta,L$, and $\|g_0\|_\infty$, but not on $n$.

Consider the functions $f\in W^{1,p}(\mathbb T_L)$ satisfying
\begin{equation}
 \|f\|_\infty\leq\|g_0\|_\infty,
 \qquad \|f_\theta\|_p\leq M',
 \qquad \osc f=\osc g_0.
 \label{eq: 2.25}
\end{equation}
If $1<p<\infty$, the one-dimensional estimate
$$
 |f(\theta)-f(\eta)|
 \leq\|f_\theta\|_p|\theta-\eta|^{1-1/p}
$$
provides a common modulus of continuity.  If $p=\infty$, the functions are uniformly Lipschitz.  Thus this family has compact closure in $C^0$.  Uniform convergence preserves the oscillation, so every member of the closure still has oscillation $\osc g_0>0$.

The map $f\mapsto\partial_\theta A^{-1}f$ from $C^0(\mathbb T_L)$ to $L^2(\mathbb T_L)$ is continuous by the Green-kernel representation.  Its value can vanish only when $A^{-1}f$ is a constant, in which case $f=A(A^{-1}f)$ is also a constant.  Since no element of the compact closure is constant, there is $c>0$ such that $\|\partial_\theta A^{-1}f\|_2\geq c$ throughout this closure.

By \eqref{eq: 2.6} and \eqref{eq: 2.24}, every $g(t_n+\sigma)$ with $|\sigma|\leq\delta$ satisfies \eqref{eq: 2.25}. Since $\partial_\theta A^{-1}g=G_\theta$, we obtain
\begin{equation}
 \|G_\theta(t_n+\sigma)\|_2\geq c\quad\text{for}
 \quad |\sigma|\leq\delta.
 \label{eq: 2.26}
\end{equation}
Pass to a subsequence satisfying $t_{n+1}-t_n>2\delta$.  Then the intervals $[t_n-\delta,t_n+\delta]$ are disjoint, so we have
$$
 \int_0^\infty\|G_\theta(t)\|_2^2\dd t
 \geq\sum_{n\ \mathrm{large}}
 \int_{t_n-\delta}^{t_n+\delta}c^2\dd t
 =\sum_{n\ \mathrm{large}}2\delta c^2=+\infty,
$$
contradicting \eqref{eq: 2.23}.  Therefore $\|q(t)\|_p\to\infty$ as $t\to+\infty$, which proves the first conclusion in \eqref{eq: 1.11}.

For the $L\log L$ assertion, suppose instead that there are $t_n\to+\infty$ and $M<\infty$ such that
\begin{equation}
 \int_{\mathbb T_L}\Phi(|q(t_n)|)\dd\theta\leq M.
 \label{eq: 2.27}
\end{equation}
The bound in \eqref{eq: 2.27} implies $\sup_n|\mathcal H(t_n)|<\infty$.  Since $\mathcal H$ is nondecreasing by \eqref{eq: 2.21}, it is bounded above on $[0,\infty)$.  Integrating \eqref{eq: 2.21} leads to
\begin{equation}
 \int_0^\infty\|G_\theta(t)\|_2^2\dd t<\infty.
 \label{eq: 2.28}
\end{equation}
For $|\sigma|\leq\delta$, the relative-flow Jacobian belongs to $[e^{-\Lambda\delta},e^{\Lambda\delta}]$.  The transport identity and the elementary estimate $J\Phi(\xi/J)\leq C_\delta\Phi(\xi)$ for $e^{-\Lambda\delta}\leq J\leq e^{\Lambda\delta}$ imply
\begin{equation}
 \sup_{|\sigma|\leq\delta}
 \int_{\mathbb T_L}\Phi(|q(t_n+\sigma)|)\dd\theta\leq C_\delta M.
 \label{eq: 2.29}
\end{equation}
Consequently, the family $\{g(t_n+\sigma):|\sigma|\leq\delta\}$ has a common modulus of continuity. For every measurable $E\subset\mathbb T_L$ and $K>0$, split $E$ into the regions where $|q|\leq K$ and $|q|>K$.  On the latter, $|q|\leq\Phi(|q|)/\log(e+K)$, and therefore
\begin{equation}
 \int_E|q|\dd\theta
 \leq K|E|+\frac{1}{\log(e+K)}
       \int_{\mathbb T_L}\Phi(|q|)\dd\theta.
 \label{eq: 2.30}
\end{equation}
First choosing $K$ large and then $|E|$ small makes the right-hand side uniformly small.  Hence the functions $g(t_n+\sigma)$, $|\sigma|\leq\delta$, have compact $C^0$ closure, and range conservation keeps their oscillation equal to $\osc g_0>0$.  The continuous map $f\mapsto\partial_\theta A^{-1}f$ vanishes only on constants.  Compactness therefore provides a constant $c>0$ such that
$$
 \|G_\theta(t_n+\sigma)\|_2\geq c
 \qquad(|\sigma|\leq\delta).
$$
After extracting a subsequence with $t_{n+1}-t_n>2\delta$, the corresponding time windows are disjoint and
$$
 \int_0^\infty\|G_\theta(t)\|_2^2\dd t
 \geq\sum_{n\ \mathrm{large}}
 \int_{t_n-\delta}^{t_n+\delta}c^2\dd t
 =\sum_{n\ \mathrm{large}}2\delta c^2=+\infty.
$$
This contradicts \eqref{eq: 2.28} and thus proves \eqref{eq: 1.13} in positive time.

For backward time, define
\begin{equation}
 \widetilde g(t,\theta):=-g(-t,\theta),
 \qquad
 \widetilde G(t,\theta):=-G(-t,\theta).
 \label{eq: 2.31}
\end{equation}
Direct calculation yields
$$
\begin{aligned}
 \widetilde g_t+2\widetilde G\widetilde g_\theta
 &=g_t(-t,\theta)+2[-G(-t,\theta)][-g_\theta(-t,\theta)]\\
 &=g_t(-t,\theta)+2G(-t,\theta)g_\theta(-t,\theta)=0.
\end{aligned}
$$
Moreover, $A\widetilde G=\widetilde g$.  Thus $(\widetilde g,\widetilde G)$ solves the same equation.  Applying the positive-time results to $\widetilde g$ proves both negative-time assertions for $g$.
\end{proof}

\subsection{Time translates and flow continuity}

\begin{lem}\label{lem: 2.3}
Let $g_0\in BV(\mathbb T_L)$, let $g$ be the corresponding unique complete solution of \eqref{eq: 1.4}, and let $t_n\to+\infty$ or $t_n\to-\infty$.  Then, after passing to a subsequence, there is a complete $BV$ solution $\bar g$ such that, for every $1\leq r<\infty$ and $0\leq\alpha<1/r$,
\begin{equation}
 g(t_n+\cdot)\longrightarrow\bar g
 \quad\text{in }C_{\mathrm{loc}}
 (\mathbb R;W^{\alpha,r}(\mathbb T_L)).
 \label{eq: 2.32}
\end{equation}
For every $\sigma\in\mathbb R$,
\begin{equation}
 \|\bar g(\sigma)\|_\infty\leq\|g_0\|_\infty,
 \qquad
 \|D\bar g(\sigma)\|_{\mathcal M}
 \leq\|Dg_0\|_{\mathcal M}.
 \label{eq: 2.33}
\end{equation}
In particular, throughout this range, the set $\{\bar g(\sigma):\sigma\in\mathbb R\}$ has compact closure in $W^{\alpha,r}(\mathbb T_L)$.

\end{lem}

\begin{proof}
The uniform spatial bounds \eqref{eq: 2.6} and \eqref{eq: 2.13}, together with one-dimensional $BV$ compactness, imply pointwise-in-time relative compactness in $L^1$.

For time equicontinuity, fix $t_0$ and use the flow starting at time $t_0$. For $\varphi\in C^1(\mathbb T_L)$, the characteristic formula reads
$$
 \int_{\mathbb T_L}\varphi(\theta)g(t,\theta)\dd\theta
 =\int_{\mathbb T_L}
   \varphi\bigl(\widehat\chi(t;t_0,\alpha)\bigr)g(t_0,\alpha)
   J(t;t_0,\alpha)\dd\alpha.
$$
The right-hand side is absolutely continuous in $t$.  Differentiating it, using \eqref{eq: 2.8}, changing variables, and integrating by parts in the $BV$ sense, we obtain, for $t_0<t$,
\begin{equation}
 \int_{\mathbb T_L}\varphi[g(t)-g(t_0)]\dd\theta
 =-\int_{t_0}^t\int_{\mathbb T_L}\varphi(\theta)u(\tau,\theta)
        \dd Dg(\tau)(\theta)\dd\tau.
 \label{eq: 2.34}
\end{equation}
The formula extends from $C^1$ to $C^0$ test functions by uniform approximation, since the right-hand side is bounded by the total variation of $Dg$.  The signed measure $[g(t)-g(t_0)]\dd\theta$ has total variation $\|g(t)-g(t_0)\|_{L^1}$.  Taking the supremum in \eqref{eq: 2.34} over $\|\varphi\|_\infty\leq1$ and using \eqref{eq: 2.4} and \eqref{eq: 2.13}, we obtain
\begin{equation}
 \|g(t)-g(t_0)\|_{L^1}
 \leq C_L\|g_0\|_\infty\|Dg_0\|_{\mathcal M}|t-t_0|.
 \label{eq: 2.35}
\end{equation}

Arzel\`a--Ascoli on compact time intervals and a diagonal argument yield convergence in $C_{\mathrm{loc}}(\mathbb R;\allowbreak L^1)$ to a function $\bar g$.  Set
$$
 g_n(\sigma):=g(t_n+\sigma),
 \qquad
 u_n(\sigma):=2A^{-1}g_n(\sigma).
$$
Estimate \eqref{eq: 2.5} implies
\begin{equation}
 u_n\longrightarrow\bar u:=2A^{-1}\bar g
 \quad\text{in }C_{\mathrm{loc}}(\mathbb R;C^1(\mathbb T_L)).
 \label{eq: 2.36}
\end{equation}
To pass to the equation without multiplying a limiting measure by a merely weakly convergent coefficient, we use the equation satisfied by $g_n$ from the definition of $u_n$ and $A^{-1}$
\begin{equation}
 (g_n)_t+(u_ng_n)_\theta=(u_n)_\theta g_n.
 \label{eq: 2.37}
\end{equation}
Here $u_ng_n\to\bar u\bar g$ and $(u_n)_\theta g_n\to\bar u_\theta\bar g$ locally in $L^1$.  More precisely, on every compact time interval $I$,
$$
\begin{aligned}
 \sup_{\sigma\in I}\|u_ng_n-\bar u\bar g\|_1
 &\leq L\|g_0\|_\infty
       \sup_{\sigma\in I}\|u_n-\bar u\|_\infty
   +\sup_{\sigma\in I}\|\bar u\|_\infty
       \sup_{\sigma\in I}\|g_n-\bar g\|_1,\\
 \sup_{\sigma\in I}\|(u_n)_\theta g_n-\bar u_\theta\bar g\|_1
 &\leq L\|g_0\|_\infty
       \sup_{\sigma\in I}\|(u_n)_\theta-\bar u_\theta\|_\infty
   +\sup_{\sigma\in I}\|\bar u_\theta\|_\infty
       \sup_{\sigma\in I}\|g_n-\bar g\|_1.
\end{aligned}
$$
The right-hand sides of both inequalities above tend to $0$ by \eqref{eq: 2.36}.  Therefore the conservative form \eqref{eq: 2.37} passes to distributions, and $\bar g$ solves \eqref{eq: 1.4} on every finite time interval.  Because the construction works on every compact time interval, $\bar g$ is complete. Uniqueness in \cite[Proposition~3.5]{EJ}, together with the time reversal in \eqref{eq: 2.31}, identifies $\bar g$ as the unique complete solution through $\bar g(0)$.

For $1<s<\infty$, the explicit interpolation estimate from the common $L^\infty$ bound is
\begin{equation}
 \|g_n-\bar g\|_s^s
 \leq (2\|g_0\|_\infty)^{s-1}\|g_n-\bar g\|_1.
 \label{eq: 2.38}
\end{equation}
The uniform version of \cref{lem: 2.1} proves \eqref{eq: 2.32}.  The bounds \eqref{eq: 2.33} follow from weak-* lower semicontinuity.  Finally, if $(\sigma_n)\subset\mathbb R$ is arbitrary, then the functions $\bar g(\sigma_n)$ satisfy the same uniform $BV$ and $L^\infty$ bounds. One-dimensional $BV$ compactness and \cref{lem: 2.1} therefore yield a subsequence converging in every stated $W^{\alpha,r}$.
\end{proof}

\begin{lem}\label{lem: 2.4}
Assume $f_n,f\in BV(\mathbb T_L)$, and
\begin{equation}
 \sup_n\bigl(\|f_n\|_\infty+\|Df_n\|_{\mathcal M}\bigr)<\infty,
 \qquad f_n\longrightarrow f\quad\text{in }L^1.
 \label{eq: 2.39}
\end{equation}
Let $S_t f$ denote the solution of \eqref{eq: 1.4} at time $t$ with initial value $f$, then, for every $T>0$ and $1\leq r<\infty$, $0\leq\alpha<1/r$, we have
\begin{equation}
 S_tf_n\longrightarrow S_tf
 \quad\text{in }C([-T,T];W^{\alpha,r}).
 \label{eq: 2.40}
\end{equation}
\end{lem}

\begin{proof}
The estimates \eqref{eq: 2.6}, \eqref{eq: 2.13}, and \eqref{eq: 2.35} provide uniform $L^\infty$, $BV$, and time equicontinuity bounds, depending only on the bound in \eqref{eq: 2.39} and on $T$.  Hence every subsequence has a further subsequence, denoted again by $f_n$, for which
\begin{equation}
 h_n(t):=S_tf_n\longrightarrow h(t)
 \quad\text{in }C([-T,T];L^1).
 \label{eq: 2.41}
\end{equation}
By \eqref{eq: 2.5}, the corresponding velocities $u_n(t):=2A^{-1}h_n(t)$ satisfy
\begin{equation}
 \sup_{|t|\leq T}\|u_n(t)-u(t)\|_{C^1}
 \leq C_L\sup_{|t|\leq T}\|h_n(t)-h(t)\|_1\longrightarrow0,
 \qquad u:=2A^{-1}h.
 \label{eq: 2.42}
\end{equation}
Rewriting the equation satisfied by $h_n$ as follows
$$
 (h_n)_t+(u_nh_n)_\theta=(u_n)_\theta h_n
$$
and using \eqref{eq: 2.41}--\eqref{eq: 2.42} together with the common $L^\infty$ bound, we obtain $u_nh_n\to uh$ and $(u_n)_\theta h_n\to u_\theta h$ in $C([-T,T];L^1)$.  Hence the equations pass to the distributional limit and $h$ solves \eqref{eq: 1.4}.  The identity $h(0)=f$ follows from $f_n\to f$ in $L^1$.  Uniqueness in \cite[Proposition~3.5]{EJ}, together with \eqref{eq: 2.31}, identifies $h(t)$ with $S_tf$.  Since every subsequence has a further subsequence with this limit, the full sequence converges in $C([-T,T];L^1)$.

Finally, the uniform $BV$ and $L^\infty$ bounds and the uniform part of \cref{lem: 2.1} imply \eqref{eq: 2.40}.
\end{proof}

\section{Atomicity and relaxation to jump profiles}
\label{sec: 3}

In this section, we prove Theorem \ref{thm: 1.3}, Theorem \ref{thm: 1.4}, Corollary \ref{cor: 1.5} and Corollary \ref{cor: 1.7}.  We first establish the preliminary results in subsections 3.1--3.3, and then give the proof of Theorem \ref{thm: 1.3} at the end of subsection 3.3. The strict inequality $L<\pi$ permits a real projective chart containing any finite set of material labels that lies strictly inside an interval on which $Dg$ has one sign.  The computations are local in time and agree on overlapping charts.  Every circular formula uses a single ordered lift of the quadruple rather than independently shifted representatives.  This convention also covers such an interval equal to the circle minus a cut point.

\subsection{Cross-ratio monotonicity on intervals of one sign}

Fix an oriented arc of length less than $\pi$.  Choose $c$ so that the closed arc does not meet $c+\pi/2$ modulo $\pi$, and set
\begin{equation}
 z=\tan(\theta-c),
 \qquad
 w(\theta)=\cos^2(\theta-c),
 \qquad
 U(t,z)=\frac{u(t,\theta(z))}{w(\theta(z))}.
 \label{eq: 3.1}
\end{equation}
Since
$$
 \frac{\dd z}{\dd\theta}=\sec^2(\theta-c)=\frac1w,
$$
we have
\begin{equation}
 D_z=wD_\theta.
 \label{eq: 3.2}
\end{equation}
The definition of $U$ is also dynamically natural: if $z(t)=\tan(\theta(t)-c)$ and $\dot\theta=u(t,\theta)$, then we have
$$
 \dot z=\frac1w\dot\theta=\frac uw=U(t,z).
$$

To compute the first three partial derivatives of $ U $ with respect to $z$, note that $z=\tan(\theta-c)$ and $w_\theta=-2wz$.  Using \eqref{eq: 3.2}, we obtain
\begin{align}
 U_z
 &=wD_\theta\left(\frac uw\right)
 =w\left(\frac{u_\theta}{w}-\frac{uw_\theta}{w^2}\right)
 =u_\theta+2zu.
 \label{eq: 3.3}
\end{align}
For the second derivative, use $z_\theta=\sec^2(\theta-c)=1/w$:
\begin{align}
 U_{zz}
 &=wD_\theta(u_\theta+2zu)
 \notag\\
 &=w\left(u_{\theta\theta}+\frac{2}{w}u+2zu_\theta\right)
 \notag\\
 &=wu_{\theta\theta}+2u
   +2\sin(\theta-c)\cos(\theta-c)u_\theta.
 \label{eq: 3.4}
\end{align}
Let $s_c=\sin(\theta-c)$ and $c_c=\cos(\theta-c)$.  Then $w=c_c^2$, $w_\theta=-2s_cc_c$, and $(s_cc_c)_\theta=c_c^2-s_c^2$.  Differentiating \eqref{eq: 3.4} before applying the final factor $w$ yields
$$
\begin{aligned}
 D_\theta(wu_{\theta\theta})
 &=-2s_cc_cu_{\theta\theta}+wu_{\theta\theta\theta},\\
 D_\theta(2u)&=2u_\theta,\\
 D_\theta(2s_cc_cu_\theta)
 &=2(c_c^2-s_c^2)u_\theta+2s_cc_cu_{\theta\theta}.
\end{aligned}
$$
The two $u_{\theta\theta}$ terms cancel.  Moreover,
$$
 2+2(c_c^2-s_c^2)=2+2\cos(2(\theta-c))=4c_c^2=4w.
$$
Therefore
\begin{align}
 U_{zzz}
 &=w\left(wu_{\theta\theta\theta}+4wu_\theta\right)
 =w^2(u_{\theta\theta\theta}+4u_\theta)
 =2w^2g_\theta,
 \label{eq: 3.5}
\end{align}
where the last equality uses the second equation in \eqref{eq: 1.6}.

On the chosen projective interval, the corresponding distributional identity is
\begin{equation}
 D_z^3U
 =z_\#\bigl(2\cos^2(\theta-c)D_\theta g\bigr).
 \label{eq: 3.6}
\end{equation}
To check the weight, in the absolutely continuous case $\dd z=w^{-1}\dd\theta$, and the density of the right-hand side with respect to $\dd z$ is
$$
 (2wg_\theta)\frac{\dd\theta}{\dd z}=2w^2g_\theta,
$$
which is exactly \eqref{eq: 3.5}. For general $BV$ data, we will apply a smoothing process to show that \eqref{eq: 3.6} is true. Fix a time and let $g_\varepsilon=\rho_\varepsilon*g$, where $\rho_\varepsilon$ is a smooth periodic approximate identity.  Next, we set
$$
u_\varepsilon:=2A^{-1}g_\varepsilon,
\qquad
U_\varepsilon(z):=
\frac{u_\varepsilon(\theta(z))}{w(\theta(z))}.
$$
Then $g_\varepsilon\to g$ in $L^1(\mathbb T_L)$ and
$$
Dg_\varepsilon\stackrel{*}{\rightharpoonup}Dg
\qquad\text{in }\mathcal M(\mathbb T_L).
$$
In addition, by \eqref{eq: 2.5}, we obtain
$$
u_\varepsilon\longrightarrow u
\qquad\text{in }C^1(\mathbb T_L).
$$
Hence, on every compact subchart on which $w$ is bounded away from zero, $U_\varepsilon\to U$ in $C^1$.  Since $g_\varepsilon$ is smooth, \eqref{eq: 3.5} implies that
$$
D_z^3U_\varepsilon
=z_\#\bigl(2wD_\theta g_\varepsilon\bigr)
$$
on this subchart.  Therefore, for $\psi\in C_c^\infty$ supported in the chosen $z$-chart, we have
$$
\begin{aligned}
\langle D_z^3U_\varepsilon,\psi\rangle
&=2\int_{\mathbb T_L}
   \psi(z(\theta))w(\theta)\,\dd Dg_\varepsilon(\theta)\\
&\longrightarrow
2\int_{\mathbb T_L}
   \psi(z(\theta))w(\theta)\,\dd Dg(\theta),
\end{aligned}
$$
where the convergence follows from the weak-* convergence of $Dg_\varepsilon$.  On the other hand, $U_\varepsilon\to U$ locally uniformly in the $z$-chart, and hence $D_z^3U_\varepsilon\to D_z^3U$ in distributions, which proves \eqref{eq: 3.6}.  In particular, $D_z^3U$ has the same sign as $D_\theta g$ on the chosen arc.

Let four material particles in one projective lift satisfy
\begin{equation}
 x_1<x_2<x_3<x_4,
 \qquad
 \dot x_i=U(t,x_i).
 \label{eq: 3.7}
\end{equation}
Define
\begin{equation}
 R(x_1,x_2,x_3,x_4)
 :=\frac{(x_2-x_1)(x_4-x_3)}
         {(x_3-x_1)(x_4-x_2)}\in(0,1).
 \label{eq: 3.8}
\end{equation}
Differentiating each logarithm yields
\begin{align}
 \frac{\dd}{\dd t}\log R
 &=\frac{U(x_2)-U(x_1)}{x_2-x_1}
   +\frac{U(x_4)-U(x_3)}{x_4-x_3}
 \notag\\
 &\quad-\frac{U(x_3)-U(x_1)}{x_3-x_1}
   -\frac{U(x_4)-U(x_2)}{x_4-x_2}
 =:L_xU.
 \label{eq: 3.9}
\end{align}

Set
\begin{equation}
 A_1:=x_2-x_1,
 \qquad B_1:=x_3-x_2,
 \qquad C_1:=x_4-x_3.
 \label{eq: 3.10}
\end{equation}
In the representation $L_xU=\sum_{i=1}^4\ell_iU(x_i)$, the coefficients of the four values $U(x_i)$ in \eqref{eq: 3.9} are
\begin{align}
 \ell_1&=-\frac1{A_1}+\frac1{A_1+B_1}
        =-\frac{B_1}{A_1(A_1+B_1)},
 \label{eq: 3.11}\\
 \ell_2&=\frac1{A_1}+\frac1{B_1+C_1},
 \label{eq: 3.12}\\
 \ell_3&=-\frac1{C_1}-\frac1{A_1+B_1},
 \label{eq: 3.13}\\
 \ell_4&=\frac1{C_1}-\frac1{B_1+C_1}
        =\frac{B_1}{C_1(B_1+C_1)}.
 \label{eq: 3.14}
\end{align}
The coefficients in \eqref{eq: 3.11}--\eqref{eq: 3.14} satisfy
\begin{equation}
 \sum_i\ell_i=0,
 \qquad
 \sum_i\ell_ix_i=0,
 \qquad
 \sum_i\ell_ix_i^2=0.
 \label{eq: 3.15}
\end{equation}
Thus $L_x$ annihilates every polynomial of degree at most two.

For a smooth $U$, Taylor's formula with integral remainder about $x_1$ is
\begin{equation}
 U(x)=U(x_1)+U'(x_1)(x-x_1)
 +\frac12U''(x_1)(x-x_1)^2
 +\int_{x_1}^x\frac{(x-s)^2}{2}U'''(s)\dd s.
 \label{eq: 3.16}
\end{equation}
Applying $L_x$ to \eqref{eq: 3.16} and using \eqref{eq: 3.15} eliminates the first three terms.  Interchanging the finite sum and the integral yields
\begin{equation}
 L_xU=\int_{x_1}^{x_4}P_x(s)U'''(s)\dd s,
 \qquad
 P_x(s):=\frac12\sum_{i:x_i\geq s}\ell_i(x_i-s)^2.
 \label{eq: 3.17}
\end{equation}
For $D_z^3U$ a finite Radon measure, smooth convolution and continuity of $P_x$ lead to the measure-level identity
\begin{equation}
 \frac{\dd}{\dd t}\log R
 =\int_{[x_1,x_4]}P_x(s)\,\dd(D_z^3U)(s).
 \label{eq: 3.18}
\end{equation}

The kernel is explicit on each of the three subintervals.  For $x_3\leq s\leq x_4$, \cref{eq: 3.17,eq: 3.14} yield
\begin{equation}
 P_x(s)=\frac{B_1(x_4-s)^2}{2C_1(B_1+C_1)}.
 \label{eq: 3.19}
\end{equation}
If $x_2\leq s\leq x_3$, the terms $i=3,4$ contribute:
\begin{equation}
 P_x(s)=\frac12\left[
 \left(-\frac1{C_1}-\frac1{A_1+B_1}\right)(x_3-s)^2
 +\left(\frac1{C_1}-\frac1{B_1+C_1}\right)(x_4-s)^2
 \right].
 \label{eq: 3.20}
\end{equation}
Finally, on $x_1\leq s\leq x_2$, the terms $i=2,3,4$ contribute.  Use \eqref{eq: 3.15} in the definition \eqref{eq: 3.17} to obtain
$$
\begin{aligned}
 P_x''(s)&=\ell_2+\ell_3+\ell_4=-\ell_1
 =\frac{B_1}{A_1(A_1+B_1)},\\
 P_x(x_1)&=\frac12\sum_i\ell_i(x_i-x_1)^2=0,\\
 P_x'(x_1)&=-\sum_i\ell_i(x_i-x_1)=0,
\end{aligned}
$$
where the last two equalities are the corresponding moment cancellations.  Two integrations therefore yield
\begin{equation}
 P_x(s)=\frac{B_1(s-x_1)^2}{2A_1(A_1+B_1)}
 \qquad(x_1\leq s\leq x_2).
 \label{eq: 3.21}
\end{equation}
It remains to determine the sign of this kernel.  The two outer pieces are nonnegative and positive in their interiors.  On the middle interval, differentiating \eqref{eq: 3.20} twice shows that
\begin{equation}
 P_x''(s)=-\frac1{A_1+B_1}-\frac1{B_1+C_1}<0.
 \label{eq: 3.22}
\end{equation}
At the endpoints, the kernel satisfies
\begin{equation}
 P_x(x_2)=\frac{A_1B_1}{2(A_1+B_1)}>0,
 \qquad
 P_x(x_3)=\frac{B_1C_1}{2(B_1+C_1)}>0.
 \label{eq: 3.23}
\end{equation}
A concave function lies above its chord, so it is positive throughout $[x_2,x_3]$.  We have proved that
\begin{equation}
 P_x(s)\geq0\quad(x_1\leq s\leq x_4),
 \qquad
 P_x(s)>0\quad(x_1<s<x_4).
 \label{eq: 3.24}
\end{equation}
The outer formulas \eqref{eq: 3.19} and \eqref{eq: 3.21} also show that $P_x=P_x'=0$ at $x_1,x_4$.  Hence the extension of $P_x$ by zero outside $[x_1,x_4]$ is continuously differentiable.

Fix a cyclically ordered material quadruple and choose its ordered lift
$$
 \theta_1<\theta_2<\theta_3<\theta_4<\theta_1+L.
$$
This ordered lift is unique up to adding the same integer multiple of $L$ to all four entries.  For $z_i=\tan(\theta_i-c)$ in any projective chart containing the lifted arc,
$$
 z_j-z_i
 =\frac{\sin(\theta_j-\theta_i)}
        {\cos(\theta_j-c)\cos(\theta_i-c)}.
$$
Substitution into \eqref{eq: 3.8} cancels every cosine factor and leaves the chart-independent formula
\begin{equation}
 R(\theta_1,\theta_2,\theta_3,\theta_4)
 =\frac{
 \sin(\theta_2-\theta_1)\sin(\theta_4-\theta_3)}
 {\sin(\theta_3-\theta_1)\sin(\theta_4-\theta_2)}.
 \label{eq: 3.25}
\end{equation}
This agrees with the cross-ratio in \eqref{eq: 3.8} in every admissible chart.  The expression is unchanged by a common lift translation since it contains only differences.  Independent changes $\theta_i\mapsto\theta_i+k_iL$ are not allowed and, since $L\neq\pi$, would generally change the sine factors. All differences in \eqref{eq: 3.25} belong to $(0,L)$, so the four sine factors are positive.  Thus the quantity is well defined once the common ordered lift has been fixed.

Suppose the four labels lie strictly inside a material arc on which $D_\theta g$ is a positive measure.  Combining \eqref{eq: 3.6}, \eqref{eq: 3.9}, \eqref{eq: 3.18}, and \eqref{eq: 3.24}, we obtain
\begin{equation}
 \frac{\dd}{\dd t}\log R\geq0.
 \label{eq: 3.26}
\end{equation}
On a material arc where $-D_\theta g$ is a positive measure,
\begin{equation}
 \frac{\dd}{\dd t}\log R\leq0,
 \label{eq: 3.27}
\end{equation}
and therefore
\begin{equation}
 \frac{\dd}{\dd t}\log(1-R)
 =-\frac{R}{1-R}
   \frac{\dd}{\dd t}\log R\geq0.
 \label{eq: 3.28}
\end{equation}
The inequality in \eqref{eq: 3.26} is strict if $D_\theta g$ assigns positive mass to the open material subarc between the first and fourth labels.  The inequality in \eqref{eq: 3.28} is strict if $-D_\theta g$ assigns positive mass to the same open subarc, since the weight in \eqref{eq: 3.6} is positive and $P_x>0$ in the open interval.

At a fixed time, the closed four-label arc has length strictly smaller than $L<\pi$, so a projective point at infinity can be chosen away from it. By continuity the same chart remains valid on a short time interval. Since \eqref{eq: 3.25} is independent of the chart, the local derivative formulas agree on overlaps and yield the global monotonicities \eqref{eq: 3.26} and \eqref{eq: 3.28}.

\subsection{Limits of transported measures and quantiles}

For the fixed decomposition \eqref{eq: 1.16}, retain the individual transported measures and atom trajectories:
\begin{equation}
 \nu_j^\pm(t)=\chi(t)_\#\mu_j^\pm,
 \qquad
 b_k(t)=\chi(t,b_k).
 \label{eq: 3.29}
\end{equation}
Convergence of the total functions alone does not show which part of the limiting signed measure comes from each initial sign component. The following compactness lemma applies to any finite family of transported positive measures and material labels.

\begin{lem}\label{lem: 3.1}
Let $t_n\to+\infty$ or $t_n\to-\infty$.  Fix finite positive measures $\mu_1,\ldots,\mu_N$ on $\mathbb T_L$ and a possibly empty finite list of material labels $\alpha_1,\ldots,\alpha_\ell$, and set $\nu_i(t):=\chi(t)_\#\mu_i$.  Then after passing to a common subsequence, there is a complete solution $\bar g$ of \eqref{eq: 1.4}, for which the following statements hold simultaneously:
\begin{enumerate}[label=(\roman*)]
\item $g(t_n+\cdot)\to\bar g$ in $C_{\mathrm{loc}}(\mathbb R;L^s)$ for every $1\leq s<\infty$;
\item for $1\leq i\leq N$,
\begin{equation}
 \nu_i(t_n+\cdot)\longrightarrow\bar\nu_i(\cdot)
 \quad\text{locally uniformly in }d_{\mathrm{BL}};
 \label{eq: 3.30}
\end{equation}
\item for $1\leq j\leq\ell$, the trajectories $\chi(t_n+\cdot,\alpha_j)$ converge locally uniformly as maps into $\mathbb T_L$.
\end{enumerate}
If $\bar\chi(\sigma;0,\cdot)$ is the flow of $\bar u=2A^{-1}\bar g$, then we have
\begin{equation}
 \bar\nu_i(\sigma)
 =\bar\chi(\sigma;0)_\#\bar\nu_i(0),
 \qquad 1\leq i\leq N.
 \label{eq: 3.31}
\end{equation}
\end{lem}

\begin{proof}
Apply \cref{lem: 2.3} to the total functions.  By \eqref{eq: 2.36},
$$
 u_n(\sigma,\theta):=u(t_n+\sigma,\theta)
 \longrightarrow\bar u(\sigma,\theta)
 \quad\text{in }C_{\mathrm{loc}}(\mathbb R;C^1).
$$
Let
$$
 \chi_n(\sigma,x):=\chi(t_n+\sigma;t_n,x)
$$
be the relative flow.  Subtracting the ODEs for $\chi_n$ and $\bar\chi$, using the bound $\|(u_n)_\theta\|_\infty\leq\Lambda$, and applying Gronwall's inequality yields the explicit estimate
\begin{equation}
 \sup_{|\sigma|\leq T}\|\chi_n(\sigma)-\bar\chi(\sigma;0)\|_\infty
 \leq e^{\Lambda T}\int_{-T}^{T}
       \|u_n(\rho)-\bar u(\rho)\|_\infty\dd\rho
 \longrightarrow0.
 \label{eq: 3.32}
\end{equation}
The spatial derivative of $\chi_n$ solves
$$
 \partial_\sigma D_x\chi_n
 =(u_n)_\theta(\sigma,\chi_n)D_x\chi_n,
 \qquad D_x\chi_n(0)=1.
$$
The derivative $D_x\bar\chi$ satisfies the analogous equation with $\bar u$. By \eqref{eq: 3.32} and the uniform continuity of $\bar u_\theta$ on $[-T,T]\times\mathbb T_L$,
$$
\begin{aligned}
 &\sup_{|\sigma|\leq T}
 \|(u_n)_\theta(\sigma,\chi_n(\sigma))
   -\bar u_\theta(\sigma,\bar\chi(\sigma;0))\|_\infty\\
 &\quad\leq
 \sup_{|\sigma|\leq T}\|(u_n)_\theta(\sigma)-\bar u_\theta(\sigma)\|_\infty
 +\sup_{|\sigma|\leq T}
  \|\bar u_\theta(\sigma,\chi_n(\sigma))
    -\bar u_\theta(\sigma,\bar\chi(\sigma;0))\|_\infty
 \longrightarrow0.
\end{aligned}
$$
A second Gronwall estimate therefore yields
\begin{equation}
 \sup_{|\sigma|\leq T}\|\chi_n(\sigma,\cdot)
       -\bar\chi(\sigma;0,\cdot)\|_{C^1}\longrightarrow0.
 \label{eq: 3.33}
\end{equation}

Since the family is finite and all masses are fixed, after a further extraction we have
\begin{equation}
 \nu_i(t_n)\stackrel{*}{\rightharpoonup}\bar\nu_i(0),
 \qquad 1\leq i\leq N.
 \label{eq: 3.34}
\end{equation}
After a further extraction, there are points $\bar\alpha_1,\ldots,\bar\alpha_\ell\in\mathbb T_L$ such that
$$
 \chi(t_n,\alpha_j)\longrightarrow\bar\alpha_j,
 \qquad 1\leq j\leq\ell.
$$
The limiting label trajectories are $\bar\chi(\sigma;0,\bar\alpha_j)$.  We also set
$$
 \bar\nu_i(\sigma)
 :=\bar\chi(\sigma;0)_\#\bar\nu_i(0),
 \qquad 1\leq i\leq N.
$$

The exact transport formula is
\begin{equation}
 \nu_i(t_n+\sigma)
 =\chi_n(\sigma)_\#\nu_i(t_n),
 \qquad 1\leq i\leq N.
 \label{eq: 3.35}
\end{equation}
To see that convergence is uniform in $\sigma$, let $\|\varphi\|_\infty+\Lip\varphi\leq1$.  Add and subtract the push-forward by $\bar\chi$:
$$
\begin{aligned}
 &\left|\int\varphi\circ\chi_n(\sigma)\dd\nu_i(t_n)
       -\int\varphi\circ\bar\chi(\sigma;0)\dd\bar\nu_i(0)\right|\\
 &\quad\leq \mu_i(\mathbb T_L)
 \|\chi_n(\sigma)-\bar\chi(\sigma;0)\|_\infty\\
 &\qquad
 +\left|\int\varphi\circ\bar\chi(\sigma;0)
              \dd[\nu_i(t_n)-\bar\nu_i(0)]\right|.
\end{aligned}
$$
On $[-T,T]$, we obtain
$$
 \|\varphi\circ\bar\chi(\sigma;0)\|_\infty
 +\Lip(\varphi\circ\bar\chi(\sigma;0))
 \leq 1+e^{\Lambda T}.
$$
Taking the supremum over admissible test functions proves \eqref{eq: 3.30} uniformly on $[-T,T]$. Equation \eqref{eq: 3.31} holds by the definition of $\bar\nu_i$, and the assertion for the labels follows from \eqref{eq: 3.32}.
\end{proof}

A transported lifted arc may fill, or asymptotically fill, the circle.  An atom at the cut may then appear at both endpoints of the lift.

Recall the quotient map $\quot:\mathbb R\to\mathbb T_L$.  Let $\mu$ be a non-atomic positive measure of mass $m>0$, concentrated on $I=\quot((a,b))$, where
\begin{equation}
 a<b\leq a+L.
 \label{eq: 3.36}
\end{equation}
Let $\widehat\mu$ be the unique positive measure concentrated on $(a,b)$ such that $(\quot)_\#\widehat\mu=\mu$, and define
\begin{equation}
 \alpha(\eta):=
 \inf\{x\in[a,b]:\widehat\mu([a,x])\geq\eta\},
 \qquad 0<\eta<m.
 \label{eq: 3.37}
\end{equation}

Let $(t_n)_{n\geq1}\subset\mathbb R$ be a sequence and suppose that
$$
 \nu_n:=\chi(t_n)_\#\mu\stackrel{*}{\rightharpoonup}\nu
 \quad\text{on }\mathbb T_L.
$$
Choose $k_n\in\mathbb Z$ so that
\begin{equation}
 a_n:=\widehat\chi(t_n,a)-k_nL\in[0,L),
 \qquad
 b_n:=\widehat\chi(t_n,b)-k_nL,
 \label{eq: 3.38}
\end{equation}
and set
\begin{equation}
 \widehat\nu_n:=
 (\widehat\chi(t_n,\cdot)-k_nL)_\#\widehat\mu.
 \label{eq: 3.39}
\end{equation}
For a positive measure $\rho$ of mass $m$ on $\mathbb R$, we write
\begin{equation}
 \begin{aligned}
  F_\rho(x)&:=\rho((-\infty,x]),
  &F_\rho(x^-)&:=\rho((-\infty,x)),\\
  Q_\rho(\eta)&:=\inf\{x:F_\rho(x)\geq\eta\},
  &0&<\eta<m.
 \end{aligned}
 \label{eq: 3.40}
\end{equation}

\begin{lem}\label{lem: 3.2}
After passing to a subsequence, the following conclusions hold.
\begin{enumerate}[label=(\roman*)]
\item There exist $a_\infty\leq b_\infty\leq a_\infty+L$ and a positive measure $\widehat\nu$ such that
\begin{equation}
 a_n\to a_\infty,\qquad
 b_n\to b_\infty,\qquad
 \widehat\nu_n\stackrel{*}{\rightharpoonup}\widehat\nu
 \quad\text{on }[0,2L],
 \label{eq: 3.41}
\end{equation}
and
\begin{equation}
 \widehat\nu(\mathbb R)=m,\qquad
 \supp\widehat\nu\subset[a_\infty,b_\infty],\qquad
 (\quot)_\#\widehat\nu=\nu.
 \label{eq: 3.42}
\end{equation}

\item For every $0<\eta<m$, we have
\begin{equation}
 \widehat\chi(t_n,\alpha(\eta))-k_nL
 =Q_{\widehat\nu_n}(\eta),
 \label{eq: 3.43}
\end{equation}
and, whenever $Q_{\widehat\nu}$ is continuous at $\eta$,
\begin{equation}
 Q_{\widehat\nu_n}(\eta)\longrightarrow Q_{\widehat\nu}(\eta).
 \label{eq: 3.44}
\end{equation}

\item Any $\ell$ distinct points of $\supp\nu$ admit increasing lifts
\begin{equation}
 \xi_1<\cdots<\xi_\ell\quad\text{in }\supp\widehat\nu,
 \qquad
 \quot(\xi_i)\neq\quot(\xi_j)\quad(i\neq j),
 \label{eq: 3.45}
\end{equation}
and ordered, pairwise disjoint neighborhoods $U_i=(r_i,s_i)\cap[a_\infty,b_\infty]$ satisfying
$$
 \xi_i\in U_i,\qquad
 \widehat\nu(U_i)>0,\qquad
 \widehat\nu(\{r_i,s_i\})=0,\qquad
 \sup U_\ell-\inf U_1<L.
$$
For such neighborhoods, there are continuity points $ 0<\eta_1<\cdots<\eta_\ell<m$ of $Q_{\widehat\nu}$ such that $Q_{\widehat\nu}(\eta_i)\in U_i$.  Moreover, if $\widehat\nu(\{\xi\})>0$, then any prescribed finite number of distinct continuity points may be chosen in $ \bigl(F_{\widehat\nu}(\xi^-),F_{\widehat\nu}(\xi)\bigr), $ and $Q_{\widehat\nu}$ equals $\xi$ at all of them.
\end{enumerate}
\end{lem}

\begin{proof}
Since $\widehat\chi(t_n,\cdot)$ is strictly increasing and degree one, \eqref{eq: 3.36} implies that
\begin{equation}
 a_n<b_n\leq a_n+L.
 \label{eq: 3.46}
\end{equation}
Thus $a_n\in[0,L)$, $b_n\in(0,2L)$, and $\supp\widehat\nu_n\subset [a_n,b_n]\subset[0,2L]$.  By compactness, we may pass to a subsequence satisfying \eqref{eq: 3.41}.  Testing with the constant function $1$ preserves the total mass.  To locate the support, fix $x\notin[a_\infty,b_\infty]$ and choose an open neighborhood $V_x$ whose closure has positive distance from $[a_\infty,b_\infty]$.  Then $V_x\cap[a_n,b_n]=\varnothing$ for all large $n$, so $\widehat\nu(V_x)=0$ by the Portmanteau theorem.  Hence no such $x$ belongs to the support and $\supp\widehat\nu\subset[a_\infty,b_\infty]$.  Finally, we have
$$
 (\quot)_\#\widehat\nu_n=\nu_n,
$$
and continuity of $\quot$ implies $(\quot)_\#\widehat\nu=\nu$.

Write $F_\mu=F_{\widehat\mu}$, $F_n=F_{\widehat\nu_n}$, $Q_n=Q_{\widehat\nu_n}$, and $Q=Q_{\widehat\nu}$. The function $F_\mu$ is continuous since $\widehat\mu$ is non-atomic.  If $h_n(x)=\widehat\chi(t_n,x)-k_nL$, then $h_n$ is strictly increasing and
$$
 F_n(h_n(x))=F_\mu(x).
$$
Since $h_n$ is a strictly increasing homeomorphism,
$$
 \{y:F_n(y)\geq\eta\}
 =h_n\bigl(\{x:F_\mu(x)\geq\eta\}\bigr).
$$
Continuity and monotonicity imply $\inf h_n(E)=h_n(\inf E)$ for every nonempty set $E$ bounded from below.  Therefore
$$
 Q_{(h_n)_\#\widehat\mu}(\eta)
 =h_n\bigl(Q_{\widehat\mu}(\eta)\bigr),
$$
which proves \eqref{eq: 3.43}, including when $F_\mu$ is constant on an interval.  Nonatomicity also shows that $\eta_1<\eta_2$ implies $\alpha(\eta_1)<\alpha(\eta_2)$: the equality would force a jump of $F_\mu$ at the common quantile.

Next, we prove \eqref{eq: 3.44}.  Fix a point $\eta$ at which $Q$ is continuous and let $\varepsilon>0$.  Choose
$$
 Q(\eta)-\varepsilon<x_-<Q(\eta)<x_+<Q(\eta)+\varepsilon,
$$
so that $\widehat\nu(\{x_-\})=\widehat\nu(\{x_+\})=0$ and
\begin{equation}
 \widehat\nu(( -\infty,x_-])<\eta
 <\widehat\nu(( -\infty,x_+]).
 \label{eq: 3.47}
\end{equation}
Such points can be chosen since $Q$ is continuous at $\eta$.  Weak-* convergence at the two continuity points preserves the strict inequalities in \eqref{eq: 3.47} for all large $n$.  The definition of generalized inverse then yields
$$
 x_-<Q_n(\eta)\leq x_+.
$$
Letting $\varepsilon\to0$ proves \eqref{eq: 3.44}.

Note that every point of $\supp\nu$ has a representative in $\supp\widehat\nu$ since $(\quot)_\#\widehat\nu=\nu$.  If $b_\infty-a_\infty=L$, only $a_\infty$ and $b_\infty$ can project to the same circular point. When this circular point is among the prescribed points, choose only one of its available endpoint representatives.  After relabeling, the chosen representatives satisfy \eqref{eq: 3.45}, and
$$
 \xi_\ell-\xi_1<L.
$$
Since only finitely many representatives are involved, we may choose ordered, pairwise disjoint neighborhoods
$$
 U_i=(r_i,s_i)\cap[a_\infty,b_\infty]
$$
sufficiently small that
$$
 \sup U_\ell-\inf U_1<L.
$$
Their endpoints may be chosen to be nonatoms of $\widehat\nu$, while $\widehat\nu(U_i)>0$ follows from $\xi_i\in\supp\widehat\nu$.  For each such $U_i$, the mass interval
$$
 \bigl(F_{\widehat\nu}(r_i),F_{\widehat\nu}(s_i^-)\bigr)
$$
is nonempty since its length equals $\widehat\nu(U_i)>0$.  These intervals are ordered and disjoint.  Since a monotone function has at most countably many discontinuities, choose a continuity point $\eta_i$ of $Q$ in each interval.  The definition of the generalized inverse yields $Q(\eta_i)\in U_i$.  If $\xi$ is an atom, then $Q$ is identically equal to $\xi$ on
$$
 \bigl(F_{\widehat\nu}(\xi^-),F_{\widehat\nu}(\xi)\bigr).
$$
Hence any prescribed finite number of distinct continuity points of $Q$ may be chosen in this open interval.
\end{proof}

\subsection{Atomic limits and profiles with finitely many jumps}

\begin{lem}\label{lem: 3.3}
Let $t_n\to+\infty$.  Along a common subsequence furnished by \cref{lem: 3.1}, write
$$
 \nu_j^\pm(t_n+\cdot)\longrightarrow\bar\nu_j^\pm(\cdot)
 \quad\text{locally uniformly in }d_{\mathrm{BL}}
$$
for all components in $\{\mu_j^+\}_{j=1}^{M_+}\cup\{\mu_j^-\}_{j=1}^{M_-}$. Then each $\bar\nu_j^\pm(0)$ is supported on at most three circular points.
\end{lem}

\begin{proof}
For the positive case, let $j$ be any positive-component index and apply \cref{lem: 3.2} with
$$
 \mu=\mu_j^+,\qquad m=m_j^+,\qquad
 \nu_n=\nu_j^+(t_n),\qquad \nu=\bar\nu_j^+(0).
$$
The notation $\alpha$, $k_n$, $\widehat\nu_n$, and $\widehat\nu$ below refers to this application until the negative component is considered. Consider first $\bar\nu_j^+(0)$.  If its support contained four distinct circular points, \cref{lem: 3.2} would provide ordered representatives and disjoint lifted neighborhoods
\begin{equation}
 \xi_1<\xi_2<\xi_3<\xi_4,
 \qquad U_1<U_2<U_3<U_4,
 \qquad \sup U_4-\inf U_1<L,
 \label{eq: 3.48}
\end{equation}
where $\xi_i\in U_i$ and $\widehat\nu(U_i)>0$.  Choose continuity points $0<\eta_1<\cdots<\eta_4<m_j^+$ of $Q_{\widehat\nu}$ as in \cref{lem: 3.2}, set $\alpha_i=\alpha(\eta_i)$, and define $\vartheta_i:=Q_{\widehat\nu}(\eta_i)\in U_i$.  Then
\begin{equation}
 \widehat\chi(t_n,\alpha_i)-k_nL\longrightarrow\vartheta_i
 \qquad(i=1,\ldots,4).
 \label{eq: 3.49}
\end{equation}

For these four fixed material labels, define
$$
 R_+(t):=R
 \bigl(\chi(t,\alpha_1),\ldots,\chi(t,\alpha_4)\bigr).
$$
By \eqref{eq: 3.26}, $\log R_+$ is nondecreasing and bounded above by $0$, so it has a finite limit as $t\to+\infty$. Consequently, for every $T>0$,
\begin{equation}
 \log R_+(t_n+T)-\log R_+(t_n-T)\longrightarrow0.
 \label{eq: 3.50}
\end{equation}

To pass to the limit in the cross-ratio derivative, use the same lift normalization for all four labels on a fixed interval centered at $t_n$:
\begin{equation}
 \widehat\theta_{i,n}(\sigma)
 :=\widehat\chi(t_n+\sigma,\alpha_i)-k_nL.
 \label{eq: 3.51}
\end{equation}
The relative-flow estimate \eqref{eq: 3.32}, applied to the lifted ODE, and \eqref{eq: 3.49} imply
\begin{equation}
 \widehat\theta_{i,n}\longrightarrow\widehat\theta_{i,\infty}
 \quad\text{uniformly on compact $\sigma$-intervals}.
 \label{eq: 3.52}
\end{equation}
By \eqref{eq: 3.48}, choose an open interval $\widehat I\subset\mathbb R$ of length less than $L$ that contains $[\vartheta_1,\vartheta_4]$ with positive distance from its boundary.  We choose $c$ such that $\theta\mapsto\tan(\theta-c)$ has no pole on $\overline{\widehat I}$. By \eqref{eq: 3.52}, there are constants $\delta,d,M>0$ such that, for $|\sigma|\leq\delta$ and all large $n$, all four normalized labels and their limits remain in a fixed compact subinterval of $\widehat I$ and
\begin{equation}
 x_{i,n}(\sigma):=\tan(\widehat\theta_{i,n}(\sigma)-c)\in[-M,M],
 \qquad x_{i+1,n}(\sigma)-x_{i,n}(\sigma)\geq d
 \quad(i=1,2,3).
 \label{eq: 3.53}
\end{equation}
Write
$$
 x_n(\sigma):=(x_{1,n}(\sigma),\ldots,x_{4,n}(\sigma)),
 \qquad
 x_\infty(\sigma):=(x_{1,\infty}(\sigma),\ldots,x_{4,\infty}(\sigma)),
$$
where $x_{i,\infty}(\sigma)=\tan(\widehat\theta_{i,\infty}(\sigma)-c)$.

For an ordered quadruple $x$, extend $P_x$ from $[x_1,x_4]$ by zero and denote the extension by $\widetilde P_x$.  With $a_+:=\max\{a,0\}$, the moment cancellations \eqref{eq: 3.15} extend \eqref{eq: 3.17} to the global formulas
$$
 \widetilde P_x(s)=\frac12\sum_{i=1}^4\ell_i(x)(x_i-s)_+^2,
 \qquad
 \widetilde P_x'(s)=-\sum_{i=1}^4\ell_i(x)(x_i-s)_+.
$$
For $s<x_1$, the two sums vanish by the three moment cancellations, and for $s>x_4$, every positive part vanishes.  On the parameter set specified by \eqref{eq: 3.53}, the coefficients $\ell_i(x)$ are uniformly bounded and Lipschitz functions of $x$, since every denominator in \eqref{eq: 3.11}--\eqref{eq: 3.14} is bounded below in terms of $d$.  If $|a|,|b|,|s|\leq M$, then we obtain
$$
 |(a-s)_+-(b-s)_+|\leq|a-b|,
 \qquad
 |(a-s)_+^2-(b-s)_+^2|\leq4M|a-b|.
$$
Both zero extensions vanish, together with their first derivatives, outside $[-M,M]$.  For ordered quadruples $x$ and $y$, write $|x-y|:=\max_i|x_i-y_i|$.  Applying the last two inequalities term by term, we obtain
\begin{equation}
 \|\widetilde P_x-\widetilde P_y\|_{W^{1,\infty}(\mathbb R)}
 \leq C_{d,M}|x-y|,
 \label{eq: 3.54}
\end{equation}

Since $\widehat I$ has length less than $L$, the quotient map $\quot$ is injective there.  We may therefore define a function on $\mathbb T_L$ by
\begin{equation}
 \Psi_{n,\sigma}(\quot(\widehat\theta))
 :=2\cos^2(\widehat\theta-c)
   \widetilde P_{x_n(\sigma)}(\tan(\widehat\theta-c)),
 \qquad \widehat\theta\in\widehat I,
 \label{eq: 3.55}
\end{equation}
and set $\Psi_{n,\sigma}=0$ outside $\quot(\widehat I)$. The definition is unambiguous, and the two pieces agree in a neighborhood of the boundary since $\widetilde P_{x_n(\sigma)}$ is supported on the four-label interval.  The vanishing of $\widetilde P_{x_n(\sigma)}$ and its first derivative at the two outer labels makes $\Psi_{n,\sigma}$ Lipschitz.  Define $\Psi_{\infty,\sigma}$ in the same way with $x_\infty(\sigma)$.  Combining \eqref{eq: 3.52} and \eqref{eq: 3.54} yields
\begin{equation}
 \sup_{|\sigma|\leq\delta}
 \left(\|\Psi_{n,\sigma}-\Psi_{\infty,\sigma}\|_\infty
 +\Lip(\Psi_{n,\sigma}-\Psi_{\infty,\sigma})\right)
 \longrightarrow0.
 \label{eq: 3.56}
\end{equation}

The support of $\Psi_{n,\sigma}$ lies between the first and fourth labels. This closed subarc is contained in the transported arc on which $Dg$ is positive, and the restrictions of $Dg(t_n+\sigma)$ and $\nu_j^+(t_n+\sigma)$ to this subarc agree.  Define
\begin{equation}
 D_n(\sigma)
 :=\int_{\mathbb T_L}\Psi_{n,\sigma}\dd\nu_j^+(t_n+\sigma).
 \label{eq: 3.57}
\end{equation}
Equations \eqref{eq: 3.6} and \eqref{eq: 3.18} then show that $\frac{\dd}{\dd\sigma}\log R_+(t_n+\sigma)=D_n(\sigma)$ for almost every $\sigma$. Set
$$
 D_\infty(\sigma)
 :=\int_{\mathbb T_L}\Psi_{\infty,\sigma}\dd\bar\nu_j^+(\sigma).
$$
By \eqref{eq: 3.30} and \eqref{eq: 3.56},
\begin{align}
 &\sup_{|\sigma|\leq\delta}
 |D_n(\sigma)-D_\infty(\sigma)|
 \notag\\
 &\quad\leq m_j^+
  \sup_{|\sigma|\leq\delta}
  \|\Psi_{n,\sigma}-\Psi_{\infty,\sigma}\|_\infty
 +C\sup_{|\sigma|\leq\delta}
  d_{\mathrm{BL}}\bigl(\nu_j^+(t_n+\sigma),\bar\nu_j^+(\sigma)\bigr)
 \longrightarrow0.
 \label{eq: 3.58}
\end{align}
At $\sigma=0$, $\Psi_{\infty,0}$ is strictly positive on $\quot(U_2)$, and $\bar\nu_j^+(0)(\quot(U_2))\geq\widehat\nu(U_2)>0$. Therefore
\begin{equation}
 D_\infty(0)>0.
 \label{eq: 3.59}
\end{equation}
The transport formula \eqref{eq: 3.31} and the continuous dependence of $\Psi_{\infty,\sigma}$ show that $D_\infty$ is continuous.  Hence \eqref{eq: 3.58} and \eqref{eq: 3.59} imply that there exist $\kappa>0$ and $0<\delta_0\leq\delta$ such that, for all large $n$,
$$
 \log R_+(t_n+\delta_0)-\log R_+(t_n-\delta_0)
 =\int_{-\delta_0}^{\delta_0}D_n(\sigma)\dd\sigma
 \geq2\kappa\delta_0,
$$
contradicting \eqref{eq: 3.50}.  Therefore $\#\supp\bar\nu_j^+(0)\leq3$.

For the negative case, let $j$ be any negative-component index and suppose that $\bar\nu_j^-(0)$ has at least four support points.  Apply \cref{lem: 3.2} with
$$
 \mu=\mu_j^-,\qquad m=m_j^-,\qquad
 \nu_n=\nu_j^-(t_n),\qquad \nu=\bar\nu_j^-(0).
$$
From this point to the end of the proof, the notation $\alpha$, $k_n$, $\widehat\nu_n$, and $\widehat\nu$ refers to this negative component. Choose ordered representatives and continuity levels $0<\eta_1<\cdots<\eta_4<m_j^-$ as in \cref{lem: 3.2}, and set
$$
 \alpha_i:=\alpha(\eta_i),\qquad
 \vartheta_i:=Q_{\widehat\nu}(\eta_i),\qquad
 \widehat\chi(t_n,\alpha_i)-k_nL\longrightarrow\vartheta_i.
$$
Define
$$
 R_-(t):=R
 \bigl(\chi(t,\alpha_1),\ldots,\chi(t,\alpha_4)\bigr).
$$
By \eqref{eq: 3.28}, $\log(1-R_-)$ is nondecreasing and bounded above by $0$.  Hence, for every $T>0$,
\begin{equation}
 \log(1-R_-(t_n+T))-\log(1-R_-(t_n-T))\longrightarrow0.
 \label{eq: 3.60}
\end{equation}
For these labels, set
$$
 \widehat\theta_{i,n}^-(\sigma)
 :=\widehat\chi(t_n+\sigma,\alpha_i)-k_nL,
$$
The displayed convergence at $\sigma=0$ and \eqref{eq: 3.32} yield locally uniform limits $\widehat\theta_{i,\infty}^-(\sigma)$.  For some $\delta>0$, an open interval $\widehat I^-$ of length less than $L$ contains all these positions for $|\sigma|\leq\delta$ and all large $n$.  Choose $c^-$ so that $\tan(\widehat\theta-c^-)$ has no pole on $\overline{\widehat I^-}$, and set
$$
 x_{i,n}^-(\sigma):=\tan(\widehat\theta_{i,n}^-(\sigma)-c^-),
 \qquad
 x_{i,\infty}^-(\sigma):=\tan(\widehat\theta_{i,\infty}^-(\sigma)-c^-).
$$
The formula in \eqref{eq: 3.55}, with these projective coordinates and $\widehat I^-$, defines $\Psi_{n,\sigma}^-$ and $\Psi_{\infty,\sigma}^-$.  Set
$$
 D_n^-(\sigma)
 :=\int_{\mathbb T_L}\Psi_{n,\sigma}^-\dd\nu_j^-(t_n+\sigma),
 \qquad
 D_\infty^-(\sigma)
 :=\int_{\mathbb T_L}\Psi_{\infty,\sigma}^-\dd\bar\nu_j^-(\sigma).
$$
The stability estimate \eqref{eq: 3.54} then implies uniform bounded-Lipschitz convergence of the corresponding kernels.  Together, this kernel convergence and \eqref{eq: 3.30} yield
$$
 \sup_{|\sigma|\leq\delta}
 |D_n^-(\sigma)-D_\infty^-(\sigma)|\longrightarrow0.
$$
The limiting kernel at $\sigma=0$ is strictly positive on the second selected lifted neighborhood, which has positive $\bar\nu_j^-(0)$-mass.  Hence $D_\infty^-(0)>0$.  The transport formula \eqref{eq: 3.31} and the continuous dependence of the limiting kernel show that $D_\infty^-$ is continuous.  Uniform separation of the four labels implies that $R_-(t_n+\sigma)/(1-R_-(t_n+\sigma))$ is bounded above and below by positive constants on this window.  The sign in \eqref{eq: 3.6} is now negative, and \eqref{eq: 3.18} yields
$$
\begin{aligned}
 \frac{\dd}{\dd\sigma}\log R_-(t_n+\sigma)
 &=-D_n^-(\sigma),\\
 \frac{\dd}{\dd\sigma}\log(1-R_-(t_n+\sigma))
 &=-\frac{R_-(t_n+\sigma)}
          {1-R_-(t_n+\sigma)}
   \frac{\dd}{\dd\sigma}\log R_-(t_n+\sigma)\\
 &=\frac{R_-(t_n+\sigma)}
         {1-R_-(t_n+\sigma)}
   D_n^-(\sigma).
\end{aligned}
$$
Consequently, for some $\kappa,\delta_0>0$ and all large $n$,
$$
 \frac{\dd}{\dd\sigma}\log(1-R_-(t_n+\sigma))
 \geq \kappa
 \quad\text{for a.e. }|\sigma|\leq\delta_0.
$$
Integrating contradicts \eqref{eq: 3.60}.  Hence $\#\supp\bar\nu_j^-(0)\leq3$.

\end{proof}

\begin{lem}\label{lem: 3.4}
For a forward sequence $t_n\to+\infty$, every limiting positive component $\bar\nu_j^+(0)$ is a single atom of mass $m_j^+$, and every limiting negative component $\bar\nu_j^-(0)$ is supported on at most two points.
\end{lem}

\begin{proof}
By \cref{lem: 3.3}, every limiting component is supported on at most three circular points and is therefore atomic.  Both cases use the following notation.  Fix $\varepsilon\in\{+,-\}$ and an index $j$ of a component with that sign.  After passing to a further subsequence, apply \cref{lem: 3.2} with
$$
 \mu=\mu_j^\varepsilon,\qquad m=m_j^\varepsilon,\qquad
 \nu_n=\nu_j^\varepsilon(t_n),\qquad
 \nu=\bar\nu_j^\varepsilon(0).
$$
We denote the objects associated with this application by
$$
 a_\infty^\varepsilon,\ b_\infty^\varepsilon,\ k_n^\varepsilon,
 \ \alpha^\varepsilon,\ \widehat\nu_n^\varepsilon,
 \ \widehat\nu^\varepsilon,
$$
and set
$$
 Q_n^\varepsilon:=Q_{\widehat\nu_n^\varepsilon},
 \qquad
 Q^\varepsilon:=Q_{\widehat\nu^\varepsilon}.
$$
For every atom $\xi$ of $\widehat\nu^\varepsilon$, set
$$
 J_\xi^\varepsilon:=
 \bigl(F_{\widehat\nu^\varepsilon}(\xi^-),
       F_{\widehat\nu^\varepsilon}(\xi)\bigr).
$$
Then
\begin{equation}
 |J_\xi^\varepsilon|=\widehat\nu^\varepsilon(\{\xi\})>0,
 \qquad
 Q^\varepsilon(\eta)=\xi\quad(\eta\in J_\xi^\varepsilon),
 \label{eq: 3.61}
\end{equation}
and every $\eta\in J_\xi^\varepsilon$ is a continuity point of $Q^\varepsilon$.

If a circular atom lies at the cut, its circular mass is the sum of the lifted masses at $a_\infty^\varepsilon$ and $b_\infty^\varepsilon$.  We choose one endpoint with positive lifted mass and use no other representative of that circular atom.  Representatives of all other circular atoms lie in $(a_\infty^\varepsilon,b_\infty^\varepsilon)$, so the chosen representatives can be ordered with total span less than $L$.  Hence all levels used below satisfy
\begin{equation}
 \widehat\chi(t_n,\alpha^\varepsilon(\eta))-k_n^\varepsilon L
 =Q_n^\varepsilon(\eta)\longrightarrow Q^\varepsilon(\eta),
 \label{eq: 3.62}
\end{equation}
and their limiting representatives lie in one projective chart.

For the positive conclusion, take $\varepsilon=+$ and let $j$ be any positive-component index. Suppose first that $\bar\nu_j^+(0)$ has two distinct atoms.  Choose lifted representatives $\xi_{\mathrm L}<\xi_{\mathrm R}$ such that
$$
 \widehat\nu^+(\{\xi_{\mathrm L}\})>0,
 \qquad
 \widehat\nu^+(\{\xi_{\mathrm R}\})>0,
$$
and choose levels
$$
 \eta_1<\eta_2\quad\text{in }J_{\xi_{\mathrm L}}^+,
 \qquad
 \eta_3<\eta_4\quad\text{in }J_{\xi_{\mathrm R}}^+.
$$
Set $\alpha_i^+:=\alpha^+(\eta_i)$ and define
$$
 R_+(t):=R
 \bigl(\chi(t,\alpha_1^+),\ldots,\chi(t,\alpha_4^+)\bigr).
$$
In a projective chart containing the four labels, write
$$
 x_{i,n}:=
 \tan\bigl(
 \widehat\chi(t_n,\alpha_i^+)-k_n^+L-c
 \bigr).
$$
By \eqref{eq: 3.62}, for some $a<b$,
\begin{equation}
 (x_{1,n},x_{2,n},x_{3,n},x_{4,n})
 \longrightarrow(a,a,b,b).
 \label{eq: 3.63}
\end{equation}
Consequently,
$$
\begin{aligned}
 R_+(t_n)
 &=
 \frac{(x_{2,n}-x_{1,n})(x_{4,n}-x_{3,n})}
 {(x_{3,n}-x_{1,n})(x_{4,n}-x_{2,n})}
 \longrightarrow0.
\end{aligned}
$$
This contradicts \eqref{eq: 3.26}, since
$$
 R_+(t_n)\geq R_+(0)>0.
$$
Thus $\bar\nu_j^+(0)$ has exactly one support point.  Since its total mass is $m_j^+$, it is a single atom of mass $m_j^+$.

For the negative conclusion, take $\varepsilon=-$ and let $j$ be any negative-component index, using an independent application of the preceding setup.  Suppose that $\bar\nu_j^-(0)$ has three distinct atoms.  Choose lifted representatives
$$
 \xi_1<\xi_2<\xi_3,
 \qquad
 \widehat\nu^-(\{\xi_i\})>0\quad(i=1,2,3),
$$
and choose levels
$$
 \eta_1\in J_{\xi_1}^-,
 \qquad
 \eta_2<\eta_3\quad\text{in }J_{\xi_2}^-,
 \qquad
 \eta_4\in J_{\xi_3}^-.
$$
Choose a projective chart containing these four labels, with parameter $c$ as in \eqref{eq: 3.1}.  Set $\alpha_i^-:=\alpha^-(\eta_i)$ and define
$$
 R_-(t):=R
 \bigl(\chi(t,\alpha_1^-),\ldots,\chi(t,\alpha_4^-)\bigr).
$$
Then
$$
 x_{i,n}:=
 \tan\bigl(
 \widehat\chi(t_n,\alpha_i^-)-k_n^-L-c
 \bigr).
$$
By \eqref{eq: 3.62}, for some $a<b<d$,
\begin{equation}
 (x_{1,n},x_{2,n},x_{3,n},x_{4,n})
 \longrightarrow(a,b,b,d).
 \label{eq: 3.64}
\end{equation}
It follows that
$$
\begin{aligned}
 R_-(t_n)
 &=
 \frac{(x_{2,n}-x_{1,n})(x_{4,n}-x_{3,n})}
 {(x_{3,n}-x_{1,n})(x_{4,n}-x_{2,n})}
 \longrightarrow
 \frac{(b-a)(d-b)}{(b-a)(d-b)}
 =1.
\end{aligned}
$$
On the other hand, \eqref{eq: 3.28} yields
$$
 1-R_-(t_n)
 \geq1-R_-(0)>0,
$$
which is impossible.  Therefore $\#\supp\bar\nu_j^-(0)\leq2$.
\end{proof}

We now proceed to prove \cref{thm: 1.3}.
\begin{proof}[Proof of \cref{thm: 1.3}]
The subsequential classification implies distance convergence along the entire half-orbit.  First note that $\mathfrak A_\ell(m)$ is weak-* closed for every integer $\ell\geq1$.  If $\rho_n\in\mathfrak A_\ell(m)$ and $\rho_n\stackrel{*}{\rightharpoonup}\rho$, but $\rho$ had $\ell+1$ distinct support points, one could choose $\ell+1$ disjoint open neighborhoods with positive $\rho$-mass and $\rho$-null boundaries.  For large $n$, each neighborhood would have positive $\rho_n$ mass, forcing at least $\ell+1$ support points.  This is impossible.  Since the fixed-mass measure space is compact in $d_{\mathrm{BL}}$, if either forward distance in \eqref{eq: 1.22} failed to tend to zero, a sequence with distance bounded below would have a convergent subsequence contradicting \cref{lem: 3.4}.  Thus \eqref{eq: 1.22} holds along the full positive half-orbit.

Fix $1\leq r<\infty$ and $0\leq\alpha<1/r$, and put $B=\|g_0\|_\infty$.  If \eqref{eq: 1.24} failed, there would be $\varepsilon_0>0$ and $t_n\to+\infty$ such that
\begin{equation}
 \dist_{W^{\alpha,r}}\bigl(g(t_n),\mathcal P_{N_+}(B)\bigr)
 \geq\varepsilon_0.
 \label{eq: 3.65}
\end{equation}
Uniform $BV$ and $L^\infty$ bounds allow us to extract a subsequence such that
\begin{equation}
 g(t_n)\longrightarrow g_*
 \quad\text{in }W^{\alpha,r}(\mathbb T_L).
 \label{eq: 3.66}
\end{equation}
After applying \cref{lem: 3.1} to the finite family $\{\mu_j^+\}_{j=1}^{M_+}\cup\{\mu_j^-\}_{j=1}^{M_-}$ and to the labels $b_1,\ldots,b_K$, pass to a common subsequence along which all these measures and trajectories converge jointly.  For every $n$,
\begin{equation}
 Dg(t_n)=
 \sum_j\nu_j^+(t_n)-\sum_j\nu_j^-(t_n)
 +\sum_{k=1}^Kc_k\delta_{b_k(t_n)}.
 \label{eq: 3.67}
\end{equation}
Each term passes to the distributional limit.  By \cref{lem: 3.4}, the positive non-atomic components contribute at most $M_+$ limiting support points, the negative components at most $2M_-$, and the transported initial atoms at most $K$.  Cancellation can only reduce the support, so we have
\begin{equation}
 \#\supp Dg_*
 \leq K+M_++2M_-=N_+.
 \label{eq: 3.68}
\end{equation}
Note also that $\|g_*\|_\infty\leq B$.  Hence $g_*\in\mathcal P_{N_+}(B)$, which contradicts \eqref{eq: 3.65} and \eqref{eq: 3.66}.

For the time-reversed solution \eqref{eq: 2.31} and each material label $\beta$, the map $t\mapsto\chi(-t,\beta)$ is its flow, since
$$
 \partial_t\chi(-t,\beta)
 =-u(-t,\chi(-t,\beta))
 =2\widetilde G(t,\chi(-t,\beta)).
$$
At time zero,
\begin{equation}
 D\widetilde g_0=-Dg_0
 =\sum_{j=1}^{M_-}\mu_j^-
 -\sum_{j=1}^{M_+}\mu_j^+
 +\sum_{k=1}^K(-c_k)\delta_{b_k}.
 \label{eq: 3.69}
\end{equation}
Thus original negative components are positive components for $\widetilde g$, and original positive components are negative.  Applying the forward result to $\widetilde g$ proves \eqref{eq: 1.23}.  The corresponding forward jump count is $K+M_-+2M_+=N_-,$ which yields \eqref{eq: 1.25}.  The compactness and invariance of $\mathcal P_N(B)$ follow from \cref{lem: 2.2}.
\end{proof}

\subsection{Countably many components and jump profiles}
\label{subsec: 3.4}

\begin{proof}[Proof of \cref{thm: 1.4}]
Set $B:=\|g_0\|_\infty$ and fix $1\leq r<\infty$ and $0\leq\alpha<1/r$.  Recall the limit sets \eqref{eq: 1.28}--\eqref{eq: 1.29}.  Since the orbit is uniformly bounded in $BV\cap L^\infty$, \cref{lem: 2.1} shows that the $L^1$ closures of the time tails coincide with their $W^{\alpha,r}$ closures. Precompactness from \cref{lem: 2.3}, together with the nestedness of the tail closures, makes $\Omega_\pm(g_0)$ nonempty and compact.

These sets are invariant under the complete flow.  Let $h\in\Omega_+(g_0)$ and choose $t_n\to+\infty$ such that $g(t_n)\to h$ in $L^1$.  For every fixed $\tau$, \cref{lem: 2.4} yields
$$
g(t_n+\tau)=S_\tau g(t_n)\longrightarrow S_\tau h.
$$
Since $t_n+\tau\to+\infty$, it follows that $S_\tau h\in\Omega_+(g_0)$.  The same calculation with $t_n\to-\infty$ proves $S_\tau\Omega_-(g_0)\subset\Omega_-(g_0)$. Applying $S_\tau$ to the same inclusions with $-\tau$ in place of $\tau$ proves the reverse inclusions.  Hence
$$
S_\tau\Omega_\pm(g_0)=\Omega_\pm(g_0).
$$

To prove attraction, suppose, for instance, that the positive-time assertion fails.  Then there exist $\varepsilon_0>0$ and $t_n\to+\infty$ such that
$$
\dist_{W^{\alpha,r}}
 \bigl(g(t_n),\Omega_+(g_0)\bigr)\geq\varepsilon_0.
$$
By precompactness, we may pass to a subsequence such that $g(t_n)\to h$ in $W^{\alpha,r}$.  Since $t_n\to+\infty$, the definition of $\Omega_+(g_0)$ implies that $h\in\Omega_+(g_0)$, a contradiction. The same argument for a sequence $t_n\to-\infty$ proves the negative-time assertion in \eqref{eq: 1.30}.

For the componentwise assertions and the description of the limit profiles, write
$$
\nu_j^\pm(t)=\chi(t)_\#\mu_j^\pm,
\qquad
b_k(t)=\chi(t,b_k).
$$
Fix one component and an arbitrary sequence $t_n\to+\infty$.  For any finite collection of transported components, \cref{lem: 3.1} supplies a common subsequence for their measures and for the material labels used in the cross-ratio argument. On the transported subarc containing the selected labels,
$$
 Dg(t)\llcorner\chi(t,I_j^\pm)=\pm\nu_j^\pm(t).
$$
The cross-ratio identities and contradiction arguments in \cref{lem: 3.3,lem: 3.4} therefore imply that every forward limit of a fixed positive component is of the form $m_j^+\delta_x$, whereas every forward limit of a fixed negative component belongs to $\mathfrak A_2(m_j^-)$.  Positive measures of fixed mass form a weak-* compact set, and $\mathfrak A_2(m_j^-)$ is closed.  Hence any failure of the distance convergence in \eqref{eq: 1.22} would produce a limit contradicting \cref{lem: 3.4}.  Time reversal proves \eqref{eq: 1.23}.

Now let $h\in\Omega_+(g_0)$.  By the definition of $\Omega_+(g_0)$, there exists $t_n\to+\infty$ such that
$$
g(t_n)\longrightarrow h
\qquad\text{in }L^1.
$$
Choose enumerations of the at most countable sets $J_+$, $J_-$, and $J_0$. Nested applications of \cref{lem: 3.1} to the first $N$ entries of these enumerations yield a diagonal subsequence along which the following limits hold for every fixed index:
$$
\nu_j^+(t_n)\stackrel{*}{\rightharpoonup}
m_j^+\delta_{x_j^+},
\qquad
\nu_j^-(t_n)\stackrel{*}{\rightharpoonup}\rho_j^-,
\qquad
\rho_j^-\in\mathfrak A_2(m_j^-),
$$
and
$$
b_k(t_n)\longrightarrow y_k^+.
$$
For every $n$,
\begin{equation}
Dg(t_n)
=
\sum_{j\in J_+}\nu_j^+(t_n)
-\sum_{j\in J_-}\nu_j^-(t_n)
+\sum_{k\in J_0}c_k\delta_{b_k(t_n)}.
\label{eq: 3.70}
\end{equation}
Since $g(t_n)\to h$ in $L^1$, we have $Dg(t_n)\to Dh$ in distributions.

Let $F_+\subset J_+$, $F_-\subset J_-$, and $F_0\subset J_0$ be finite, and let $\Sigma_{F,n}$ denote the corresponding partial sum in \eqref{eq: 3.70}.  Its limit is
$$
\Sigma_F
=
\sum_{j\in F_+}m_j^+\delta_{x_j^+}
-\sum_{j\in F_-}\rho_j^-
+\sum_{k\in F_0}c_k\delta_{y_k^+}.
$$
For every test function $\varphi$,
$$
\begin{aligned}
\bigl|\langle Dh-\Sigma_F,\varphi\rangle\bigr|
&\leq
\|\varphi\|_\infty
\left(
\sum_{j\in J_+\setminus F_+}m_j^+
+\sum_{j\in J_-\setminus F_-}m_j^-
+\sum_{k\in J_0\setminus F_0}|c_k|
\right).
\end{aligned}
$$
Here push-forward preserves the mass of each component, and \eqref{eq: 1.27} bounds the remaining tail uniformly in $n$ for $Dg(t_n)-\Sigma_{F,n}$.  Passing to the limit proves the displayed estimate. Letting the finite sets exhaust the corresponding index sets and using \eqref{eq: 1.27} proves \eqref{eq: 1.31}. Since $\|\rho_j^-\|_{\mathcal M}=m_j^-$, the same summability condition also ensures absolute convergence in total variation.  Applying the argument to the reversed solution proves \eqref{eq: 1.32}.
\end{proof}

\begin{proof}[Proof of \cref{cor: 1.5}]
Set $B:=\|g_0\|_\infty$. Decompose the open sets
$$
\{g_{0,\theta}>0\}
=
\bigsqcup_{j\in J_+}I_j^+,
\qquad
\{g_{0,\theta}<0\}
=
\bigsqcup_{j\in J_-}I_j^-.
$$
Each is a countable disjoint union of open arcs.  Restricting $(g_{0,\theta})_+\,\dd\theta$ and $(g_{0,\theta})_-\,\dd\theta$ to these components yields \eqref{eq: 1.26} with $J_0=\varnothing$, while
$$
\sum_{j\in J_+}m_j^+
+\sum_{j\in J_-}m_j^-
=
\int_{\mathbb T_L}|g_{0,\theta}(\theta)|\dd\theta
<\infty,
$$
thus \eqref{eq: 1.27} holds.  The relaxation conclusion follows from \cref{thm: 1.4}, while the gradient conclusion follows from \cref{thm: 1.1}.

Suppose now that $|t_n|\to\infty$ and $g(t_n)\rightharpoonup h$ weakly in $L^2$.  Uniform $BV$ compactness shows that every subsequence has a further subsequence, still denoted by $t_n$, such that $g(t_n)\longrightarrow\widetilde h \qquad\text{in }L^1$.  The common $L^\infty$ bound, inherited by $\widetilde h$, implies
$$
\|g(t_n)-\widetilde h\|_{L^2}^2
\leq
2B\|g(t_n)-\widetilde h\|_{L^1}
\longrightarrow0.
$$
Since the original sequence converges weakly to $h$, we have $\widetilde h=h$.  Thus every subsequence has a further subsequence converging strongly to $h$ in $L^2$, and the full sequence therefore converges strongly in $L^2$.  Since $\mathbb T_L$ has finite measure, the convergence also holds in $L^1$, and \cref{lem: 2.1} upgrades the $L^1$ convergence to every stated $W^{\alpha,r}$.

After passing to a subsequence, either $t_n\to+\infty$ or $t_n\to-\infty$.  Hence $h$ belongs to the corresponding compact invariant limit set.  The formulas \eqref{eq: 1.31}--\eqref{eq: 1.32} show that $Dh$ is purely atomic, while invariance implies that the complete orbit through $h$ remains in this compact set.
\end{proof}

\begin{proof}[Proof of \cref{cor: 1.7}]
Set $B:=\|g_0\|_\infty$. For the finite decomposition, the attraction in \cref{thm: 1.3} and the closedness of $\mathcal P_N(B)$ from \cref{lem: 2.2} imply
\begin{equation}
\Omega_+(g_0)\subset\mathcal P_{N_+}(B),
\qquad
\Omega_-(g_0)\subset\mathcal P_{N_-}(B).
\label{eq: 3.71}
\end{equation}

Suppose that $g_0$ is a Morse function with $n$ maxima and $n$ minima. Its critical points alternate on the circle.  Between a minimum and the next maximum one has $g_{0,\theta}>0$, while between that maximum and the next minimum one has $g_{0,\theta}<0$.  Consequently,
\begin{equation}
K=0,
\qquad
M_+=M_-=n,
\qquad
N_+=N_-=3n.
\label{eq: 3.72}
\end{equation}
For each fixed $n$, persistence of nondegenerate critical points under small $C^2$ perturbations shows that this subclass is open.  The density statement for the union of all Morse subclasses is the standard Morse genericity theorem \cite{Hirsch}.

The gradient conclusion and \eqref{eq: 1.36} now follow from \cref{thm: 1.1,thm: 1.3} and \eqref{eq: 3.72}.  Finally, let $t_k\to+\infty$ or $t_k\to-\infty$ and suppose that $g(t_k)\rightharpoonup h$ weakly in $L^2$.  By uniform $BV$ compactness, every subsequence has a further subsequence converging strongly in $L^1$.  The common $L^\infty$ bound upgrades each such convergence to $L^2$, and weak convergence identifies the limit as $h$.  Hence the full sequence converges strongly to $h$ in both $L^1$ and $L^2$. Thus $h$ belongs to the corresponding limit set, and \eqref{eq: 3.71}--\eqref{eq: 3.72} imply
$$
h\in\mathcal P_{3n}(B).
$$
By \cref{lem: 2.2}, $\mathcal P_{3n}(B)$ is compact and invariant, so the complete orbit through $h$ remains in this set.
\end{proof}

\section{Solutions with finitely many jumps}
\label{sec: 4}

This section is devoted to the proof of Theorem \ref{thm: 1.8}. To this end, we first make some preparations in Subsection 4.1.
\subsection{Equations for the jump points}

Denote $\langle f\rangle:=L^{-1}\int_{\mathbb T_L}f\dd\theta$ for $f\in L^1(\mathbb T_L)$, and let $K_L$ be the Green kernel in \eqref{eq: 2.1}.  Let $\mathcal R$ be the mean-zero periodic primitive of $2K_L-1/(2L)$.  The formula for $K_L$ is the $m=3$ specialization of the kernel in \cite{EMS}. The computation below also verifies the normalization and distributional identity used in the jump equations.

For $0\leq r\leq L$, define
\begin{equation}
 \mathcal R(r):=
 \frac{\sin(2r-L)+\sin L}{4\sin L}-\frac{r}{2L},
 \label{eq: 4.1}
\end{equation}
and extend $\mathcal R$ $L$-periodically.  Since
$$
 \mathcal R(0)=\mathcal R(L)=0,
$$
the periodic extension is continuous.  Moreover, for $0\leq r\leq L$,
$$
\begin{aligned}
 \mathcal R(L-r)
 &=\frac{\sin(L-2r)+\sin L}{4\sin L}
   -\frac{L-r}{2L}\\
 &=\frac{-\sin(2r-L)+\sin L}{4\sin L}
   -\frac12+\frac r{2L}\\
 &=-\frac{\sin(2r-L)+\sin L}{4\sin L}
   +\frac r{2L}
 =-\mathcal R(r).
\end{aligned}
$$
Thus $\mathcal R$ is odd on the periodic circle and
\begin{equation}
 \int_0^L\mathcal R(r)\dd r=0.
 \label{eq: 4.2}
\end{equation}

On the open interval $(0,L)$, a direct computation yields
$$
\begin{aligned}
 \mathcal R'(r)
 =\frac{\cos(2r-L)}{2\sin L}-\frac1{2L},\quad
 \mathcal R''(r)
 =-\frac{\sin(2r-L)}{\sin L},\quad
 \mathcal R'''(r)
 =-\frac{2\cos(2r-L)}{\sin L}.
\end{aligned}
$$
Consequently, one has
\begin{equation}
 \mathcal R'''+4\mathcal R'=-\frac2L
 \qquad\text{on }(0,L).
 \label{eq: 4.3}
\end{equation}
The one-sided second derivatives at the periodic cut are
$$
 \mathcal R''(0+)=1,
 \qquad
 \mathcal R''(L-)=-1.
$$
Hence the periodic distribution $\mathcal R'''$ has an atom of mass $2$ at $0$.  Combining this jump with \eqref{eq: 4.3},
\begin{equation}
 (\partial_\theta^3+4\partial_\theta)\mathcal R
 =2\left(\delta_0-\frac1L\right).
 \label{eq: 4.4}
\end{equation}
Equivalently, $K_L$ satisfies
$$
 \mathcal R'=2K_L-\frac1{2L}.
$$
It follows from \eqref{eq: 4.4} and the Green-kernel identity that both $A\mathcal R'$ and $A(2K_L-1/(2L))$ equal $2\delta_0-2/L$.  Since $A$ is invertible, the claimed identity follows.  The normalization \eqref{eq: 4.2} fixes the additive constant in $\mathcal R$.

Let $g$ be a nonconstant step profile.  Its minimal representation has $N\geq2$ and may be written as
\begin{equation}
 \begin{gathered}
 Dg=\sum_{i=0}^{N-1}\Delta_i\delta_{a_i},
 \qquad
 \Delta_i:=g(a_i+)-g(a_i-)\neq0,\\
 \sum_{i=0}^{N-1}\Delta_i=0,
 \qquad
 0\leq a_0<L,\qquad
 a_0<a_1<\cdots<a_{N-1}<a_0+L.
 \end{gathered}
 \label{eq: 4.5}
\end{equation}
The last identity holds since a distributional derivative on a circle has zero total mass.  Averaging $G_{\theta\theta}+4G=g$ yields
$$
 \langle G\rangle=\frac14\langle g\rangle,
 \qquad
 \langle u\rangle=2\langle G\rangle=\frac12\langle g\rangle.
$$
The identities \eqref{eq: 4.2} and \eqref{eq: 4.4}, together with $\sum_i\Delta_i=0$, determine the unique velocity
\begin{equation}
 u(\theta)=\frac{\langle g\rangle}{2}
 +\sum_{j=0}^{N-1}\Delta_j\mathcal R(\theta-a_j).
 \label{eq: 4.6}
\end{equation}
By applying $\partial_\theta^3+4\partial_\theta$ to the right-hand side, we obtain
$$
\begin{aligned}
 \sum_j2\Delta_j\left(\delta_{a_j}-\frac1L\right)
 &=2\sum_j\Delta_j\delta_{a_j}
   -\frac2L\sum_j\Delta_j
 =2Dg,
\end{aligned}
$$
which is the second equation in \eqref{eq: 1.6}.  The prescribed mean $\langle u\rangle=\frac12\langle g\rangle$ fixes the remaining constant: the Fourier frequencies on $\mathbb T_L$ are $3k$, so the constant mode is the entire kernel of $\partial_\theta^3+4\partial_\theta$.

Let $a_i(t)=\widehat\chi(t,a_i(0))$ in an ordered lift.  By \eqref{eq: 2.15}, these remain the distinct jumps of the minimal representation for every finite time.  Moreover, \eqref{eq: 2.35} and \eqref{eq: 2.5} imply $u\in C(\mathbb R;C^1)$, so every $a_i$ is $C^1$ and $\dot a_i(t)=u(t,a_i(t))$ pointwise.  Since $\mathcal R(0)=0$, the exact jump ODE from \eqref{eq: 4.6} is
\begin{equation}
 \dot a_i=\frac{\langle g\rangle}{2}
 +\sum_{j=0}^{N-1}\Delta_j\mathcal R(a_i-a_j).
 \label{eq: 4.7}
\end{equation}
Mean conservation also holds directly for $BV$ data.  From the distributional equation $g_t=-uDg$ and integration by parts against the $C^1$ function $u$, we obtain, for almost every $t$,
\begin{align}
 \frac{\dd}{\dd t}\int_{\mathbb T_L}g\dd\theta
 &=-\int_{\mathbb T_L}u\dd Dg
 =\int u_\theta g\dd\theta
 \notag\\
 &=2\int G_\theta(G_{\theta\theta}+4G)\dd\theta
 \notag
 =\int\partial_\theta(G_\theta^2+4G^2)\dd\theta=0.
 \label{eq: 4.8}
\end{align}
Here $g\in L^\infty$ and $G_{\theta\theta}=g-4G\in L^\infty$, so every product belongs to $L^1$, while the last integral is the integral of a periodic weak derivative.  The mean term $\langle g\rangle/2$ in \eqref{eq: 4.7} produces only a common rotation and does not affect the interval lengths.

If $N=2$, then $\Delta_1=-\Delta_0$.  Let $x=a_1-a_0\in(0,L)$.  From \eqref{eq: 4.7} and the oddness of $\mathcal R$, we have
$$
\begin{aligned}
 \dot a_1
 &=\frac{\langle g\rangle}{2}+\Delta_0\mathcal R(x),\\
 \dot a_0
 &=\frac{\langle g\rangle}{2}+\Delta_1\mathcal R(-x)
 =\frac{\langle g\rangle}{2}+\Delta_0\mathcal R(x).
\end{aligned}
$$
Therefore
\begin{equation}
 \dot x=\dot a_1-\dot a_0=0.
 \label{eq: 4.9}
\end{equation}
By \eqref{eq: 1.8} and the mean-conservation computation above, the jump size $\Delta_0$ and the mean $\langle g\rangle$ are constant in time, and \eqref{eq: 4.9} shows that $x$ is also constant.  Hence the common velocity
$$
 c:=\frac{\langle g\rangle}{2}+\Delta_0\mathcal R(x)
$$
is constant, so we have
$$
 a_i(t)=a_i(0)+ct,
 \qquad
 g(t,\theta)=g_0(\theta-ct),\qquad i=0,1.
$$
Thus every two-jump profile is a travelling wave.

\subsection{Solutions with three jumps}

Let $\gamma_0,\gamma_1,\gamma_2$ be the constant values on the three consecutive arcs.  With the convention that $\gamma_i$ is the value immediately to the right of $a_i$, the jumps are
\begin{equation}
 \Delta_0=\gamma_0-\gamma_2,
 \qquad
 \Delta_1=\gamma_1-\gamma_0,
 \qquad
 \Delta_2=\gamma_2-\gamma_1.
 \label{eq: 4.10}
\end{equation}
Set the three interval lengths
\begin{equation}
 x=a_1-a_0,
 \qquad
 y=a_2-a_1,
 \qquad
 z=L-x-y.
 \label{eq: 4.11}
\end{equation}

\begin{prop}\label{prop: 4.1}
On the open simplex
\begin{equation}
 \Sigma:=\{(x,y,z):x,y,z>0,\ x+y+z=L\},
 \label{eq: 4.12}
\end{equation}
the equations for $x,y,z$ are
\begin{equation}
 \begin{aligned}
  \dot x&=\Delta_2F(x,y,z),
  &\dot y&=\Delta_0F(x,y,z),\\
  \dot z&=\Delta_1F(x,y,z),
  &F(x,y,z)&:=\frac{\sin x\sin y\sin z}{\sin L}>0.
 \end{aligned}
 \label{eq: 4.13}
\end{equation}
\end{prop}

\begin{proof}
From \eqref{eq: 4.7},
$$
\begin{aligned}
 \dot a_1
 &=\frac{\langle g\rangle}{2}
   +\Delta_0\mathcal R(x)-\Delta_2\mathcal R(y),\\
 \dot a_0
 &=\frac{\langle g\rangle}{2}
   -\Delta_1\mathcal R(x)-\Delta_2\mathcal R(x+y).
\end{aligned}
$$
Subtracting and using $\Delta_0+\Delta_1=-\Delta_2$, we obtain
\begin{equation}
 \dot x=\Delta_2
 [\mathcal R(x+y)-\mathcal R(x)-\mathcal R(y)].
 \label{eq: 4.14}
\end{equation}
In the bracket, the linear terms in \eqref{eq: 4.1} cancel.  The remaining numerator is
$$
\begin{aligned}
 &\sin(2x+2y-L)-\sin(2x-L)-\sin(2y-L)-\sin L\\
 &=\bigl[\sin(2x+2y-L)-\sin(2x-L)\bigr]
   -\bigl[\sin(2y-L)+\sin L\bigr]\\
 &=2\sin y\cos(2x+y-L)-2\sin y\cos(y-L)\\
 &=2\sin y\bigl[\cos(2x+y-L)-\cos(y-L)\bigr]\\
 &=-4\sin y\sin(x+y-L)\sin x\\
 &=4\sin x\sin y\sin(L-x-y).
\end{aligned}
$$
Substitution into \eqref{eq: 4.1} yields
\begin{equation}
 \mathcal R(x+y)-\mathcal R(x)-\mathcal R(y)
 =\frac{\sin x\sin y\sin(L-x-y)}{\sin L},
 \label{eq: 4.15}
\end{equation}
which proves the first equation in \eqref{eq: 4.13}.  Cyclically relabeling the three jumps yields $\dot y=\Delta_0F(x,y,z)$.  Finally,
$$
\begin{aligned}
 \dot z&=-\dot x-\dot y
 =-(\Delta_2+\Delta_0)F(x,y,z)
 =\Delta_1F(x,y,z),
\end{aligned}
$$
which proves the third equation.
\end{proof}

\begin{proof}[Proof of \cref{thm: 1.8}]
A nonconstant periodic step profile cannot have exactly one nonzero jump, since its distributional derivative has zero total mass.  The constant case is immediate, and the two-jump case follows from \eqref{eq: 4.9}.  It remains to consider a minimal three-jump profile.

Set
$$
 (x_0,y_0,z_0):=(x(0),y(0),z(0)).
$$
Introduce a new time parameter by
\begin{equation}
 \frac{\dd\sigma}{\dd t}=F(x(t),y(t),z(t)),
 \qquad \sigma(0)=0.
 \label{eq: 4.16}
\end{equation}
\begingroup
\sloppy
Since $F(x,y,z)>0$ in $\Sigma$, $\sigma$ is strictly increasing. By Proposition~\ref{prop: 4.1} and \eqref{eq: 4.16}, the reparametrized system has the solution
\par
\endgroup
\begin{equation}
 (x(\sigma),y(\sigma),z(\sigma))
 =(x_0,y_0,z_0)+\sigma(\Delta_2,\Delta_0,\Delta_1).
 \label{eq: 4.17}
\end{equation}
Thus $(x,y,z)$ traverses a line segment in $\Sigma$ as $t$ increases.

In a minimal three-jump representation, all three jumps $\Delta_i$ are nonzero.  Define the maximal interval
\begin{equation}
 I:=(\sigma_-,\sigma_+)
 :=\{\sigma\in\mathbb R:
 x_0+\sigma\Delta_2>0,
 \ y_0+\sigma\Delta_0>0,
 \ z_0+\sigma\Delta_1>0\}.
 \label{eq: 4.18}
\end{equation}
Since $\Delta_0+\Delta_1+\Delta_2=0$ and every $\Delta_i$ is nonzero, $(\Delta_2,\Delta_0,\Delta_1)$ has both positive and negative components. Hence $I$ is a bounded open interval containing $0$, and its endpoints lie on $\partial\Sigma$.  Physical time is recovered from
\begin{equation}
 t(\sigma)
 =\int_0^\sigma
 \frac{\dd\rho}{
 F(x_0+\rho\Delta_2,y_0+\rho\Delta_0,z_0+\rho\Delta_1)}.
 \label{eq: 4.19}
\end{equation}

Extend $x,y,z$ continuously to $\overline I$, and let $\sigma_*$ be either endpoint of $I$.  Suppose first that the line meets the interior of an edge. After a cyclic relabeling, we may assume that $x(\sigma_*)=0$.  Then $x(\sigma)=|\Delta_2||\sigma-\sigma_*|$ on the interior side of the endpoint, while $y(\sigma_*)$ and $z(\sigma_*)$ are positive.  It follows that
$$
 F(x(\sigma),y(\sigma),z(\sigma))
 \sim\frac{|\Delta_2|\sin y(\sigma_*)\sin z(\sigma_*)}{\sin L}
 |\sigma-\sigma_*|.
$$
If the line meets a vertex, we may assume after a cyclic relabeling that $x(\sigma_*)=y(\sigma_*)=0$ and $z(\sigma_*)=L$.  Since $x(\sigma)=|\Delta_2||\sigma-\sigma_*|$, $y(\sigma)=|\Delta_0||\sigma-\sigma_*|$, and $\sin z(\sigma)\to\sin L$ on the interior side, we have
$$
 F(x(\sigma),y(\sigma),z(\sigma))
 \sim |\Delta_2\Delta_0||\sigma-\sigma_*|^2.
$$
In both cases, there are $C,\delta>0$ such that $F(x(\sigma),y(\sigma),z(\sigma))\leq C|\sigma-\sigma_*|$ whenever $\sigma\in I$ and $0<|\sigma-\sigma_*|<\delta$.  Setting $\xi=|\sigma-\sigma_*|$, the corresponding part of \eqref{eq: 4.19} is bounded below by
$$
 \int_\varepsilon^\delta
 \frac{\dd\xi}{C\xi}
 \longrightarrow+\infty.
$$
Thus the integral in \eqref{eq: 4.19} diverges at both endpoints, and $t:I\to\mathbb R$ is an increasing bijection,
$$
 \sigma(t)\longrightarrow\sigma_-\quad(t\to-\infty),
 \qquad
 \sigma(t)\longrightarrow\sigma_+\quad(t\to+\infty).
$$
To verify convergence in the quotient $L^s$ metric, rotate each profile so that its first jump is at $0$, and denote the rotated profile by $g_\sigma^\sharp$.  If $g_{\sigma_*}^\sharp$ is the endpoint profile obtained by merging coincident jumps, then preservation of the three constant values implies, for $1\leq s<\infty$,
$$
 \|g_\sigma^\sharp-g_{\sigma_*}^\sharp\|_{L^s}^s
 \leq(2B)^s\Bigl(
 |x(\sigma)-x(\sigma_*)|
 +|x(\sigma)+y(\sigma)-x(\sigma_*)-y(\sigma_*)|
 \Bigr)\longrightarrow0.
$$
At an edge, two adjacent jump points coalesce and the corresponding jump sizes add, producing a two-jump profile. At a vertex, all jump points coincide and their total jump is zero, leaving a constant.  The strict monotonicity of $\sigma(t)$ also excludes periodic three-jump orbits.

For the final assertion, let $g$ denote the solution through $g_0$, set $B_0:=\|g_0\|_\infty$, and let $\bar g$ be an arbitrary complete solution obtained as a limit of time translates of $g$ along $t_n\to+\infty$ or $t_n\to-\infty$. Under $K=0$ and $M_+=M_-=1$, we have $N_+=N_-=3$.  The convergence in \eqref{eq: 2.32}, the relevant convergence in \eqref{eq: 1.24}--\eqref{eq: 1.25}, and the closedness of $\mathcal P_3(B_0)$ from \cref{lem: 2.2} imply that
$$
 \dist_{L^s}\bigl(\bar g(0),\mathcal P_3(B_0)\bigr)
 \leq \|\bar g(0)-g(t_n)\|_{L^s}
 +\dist_{L^s}\bigl(g(t_n),\mathcal P_3(B_0)\bigr)
 \longrightarrow0.
$$
Thus $\bar g(0)\in\mathcal P_3(B_0)$, and hence
$$
 \bar g(\tau)=S_\tau\bar g(0)
 \in\mathcal P_3(B_0),
 \qquad \tau\in\mathbb R.
$$
The orbit is therefore covered by the constant, two-jump, and genuine three-jump alternatives established above.
\end{proof}

\section*{Declarations}
\noindent\textbf{Acknowledgement}
D. Cao and J. Fan were supported by the National Key R\&D Program of China (2023YFA1010001) and the NNSF of China (Grant No.~12371212). G. Qin was supported by the National Key R\&D Program of China (Grant 2025YFA1018400) and the NNSF of China (Grant 12471190). The authors are grateful to Professor Ryan Murray for helpful discussions and comments that improved the presentation of this paper. OpenAI models were used for language editing and manuscript preparation. The mathematical ideas and calculations were developed by the authors, who take full responsibility for the manuscript.

\bigskip

\noindent\textbf{Author contributions.}
All authors contributed equally.

\bigskip

\noindent\textbf{Conflict of interest.}
On behalf of all authors, the corresponding author states that there is no conflict of interest.

\bigskip

\noindent\textbf{Data availability.}
Data availability is not applicable to this article, as no datasets were generated or analysed during the current study.

\end{document}